\documentclass[12pt,reqno]{amsart}
\usepackage{amssymb,amsmath,amsthm,newlfont,enumerate,xcolor}
\usepackage{amsthm}
\usepackage{a4wide}

\newtheorem{thm}{Theorem}[section]
\newtheorem{lem}[thm]{Lemma}
\newtheorem{cor}[thm]{Corollary}
\newtheorem{prop}[thm]{Proposition}
\newtheorem{defi}[thm]{Definition}

\newcommand{\TT}{\mathbb{T}}
\newcommand{\DD}{\mathbb{D}}

\newcommand{\RR}{\mathbb{R}}
\newcommand{\T}{\mathbb{T}} 
 
\newcommand{\R}{\mathbb{R}} 
\newcommand{\C}{\mathbb{C}} 

\newcommand{\D}{\mathbb{D}}

\newcommand{\cD}{\mathcal{D}}
\newcommand{\cM}{\mathcal{M}}

\newcommand{\re}{\mathrm{Re}}

\newcommand{\dist}{\mathrm{d}} 
\newcommand{\Dim}{\mathrm{dim}}
\newcommand{\dd}{\mathrm{d}} 
\begin{document}

\title{The local Dirichlet integral and applications}

\author[]{{Omar} {El-Fallah}}
\address{Laboratoire Analyse et Applications. Universit\'e Mohamed V Rabat. Morocco}
\email[O. El-Fallah]{o.elfallah@um5r.ac.ma}

\author[]{{Karim} {Kellay}} 
\address{Univ. Bordeaux, CNRS, Bordeaux INP, IMB, UMR 5251, F-33400 Talence, France}
\email[K. Kellay]{kkellay@math.u-bordeaux.fr}

\author[]{{Houssame} {Mahzouli}}
\address{Laboratoire Analyse et Applications. Universit\'e Mohamed V Rabat. Morocco}
\email[H. Mahzouli]{mahzouli@fsr.ac.ma}

\begin{abstract}
We study the local Dirichlet integral of distance functions and their behavior within the harmonic Dirichlet space.  We provide estimates for the local Dirichlet integral of distance functions, which allow us to study their membership in the algebra of multipliers of the Dirichlet space.  
We give sufficient condition for a closed subset of the unit circle to be polar and  we also examine cyclicity in the harmonic Dirichlet spaces.
\end{abstract}

\maketitle
\section{Introduction and main results}

In this work, we study the local Dirichlet integral in relation to distance functions. Our main goal is to analyze their membership in the algebra of multipliers of the classical Dirichlet space and their role in the study of cyclicity in the harmonic Dirichlet space.

Let \( f \) be a function in \( L^1(\TT) \), where  \( \mathbb{T} \) denotes the unit circle. We assume that  the nontangential limit of the Poisson transform of $f$, \( f^*(\zeta) \), exists.
The local Dirichlet integral of \( f \) is defined as  
\[
\cD_\zeta(f) := \frac{1}{2\pi} \int_{\mathbb{T}} \frac{|f^*(\xi)-f^*(\zeta)|^2}{|\xi-\zeta|^2} |d\xi|.
\]
Note that  \( \mathcal{D}_\zeta(f) \) may be infinite. If \( f^*(\zeta) \) does not exist, we set $\cD_\zeta (f) = \infty$. The formula was introduced by J.Douglas  in his study of minimal surfaces, known as the Plateau problem, within the classical Dirichlet space. It was later studied by Richter and Sundberg \cite{RS}, who extended it to account for the canonical factorization of \( f \), thereby generalizing a result of L. Carleson \cite{carleson1960representation}. Additionally, they characterized the shift-invariant subspaces of the harmonic Dirichlet space \cite{RStrans}.

To present our results, we first introduce some notation. Let \( \mu \) be a positive Borel measure on \( \mathbb{T} \). The harmonic Dirichlet space \( \mathcal{D}(\mu) \) consists of all functions \( f \) in the Hardy space \( H^2 \) for which the local Dirichlet integral is integrable with respect to \( \mu \). This space is equipped with the norm  
\[
\|f\|^2_{\mu} := \|f\|^2_{H^2} +\cD_\mu(f), 
\]
where 
$$\cD_\mu(f):= \int_\mathbb{T} \cD_\zeta(f) \, d\mu(\zeta).$$
Note that \( \mathcal{D}(\mu) \subset H^2 \), so the radial limit \( f^*(\zeta) = \lim_{r \to 1^-} f(r\zeta) \) exists almost everywhere on \( \mathbb{T} \). For simplicity, we will write $f$ instead of  \( f^* \).
When \( \dd\mu(\zeta) = {\dd m(\zeta)}/{2\pi} \), the normalized arc measure on \( \mathbb{T} \), \( \mathcal{D}(\mu) \) reduces to the classical Dirichlet space \( \mathcal{D} \), consisting of all holomorphic functions \( f \) on the unit disc \( \mathbb{D} \) whose derivative \( f' \) is square-integrable with respect to the area measure. The characterization of invariant subspaces of the Hardy space was given by Beurling \cite{B}; for Dirichlet spaces, however, the situation is different. The spaces \(  \mathcal{D}(\mu) \) were introduced by Richter \cite{Richertarns} in his study of \( 2 \)-isometries and the characterization of shift-invariant subspaces of the Dirichlet space; see also \cite{RichterCreel, RStrans, el2014primer}. {  More precisely, consider the shift-invariant subspace  $\mathcal{M}$ of  Dirichlet space $\mathcal{D}$, 
Richter established the following key properties:
\begin{itemize}
    \item \( \dim (\mathcal{M} \ominus z \mathcal{M}) = 1 \).
    \item \( \mathcal{M} = \varphi \mathcal{D}(\mu_\varphi) \), where \( \varphi \in \mathcal{M} \ominus z \mathcal{M} \),  \( \| \varphi \|_{\mathcal{D}} = 1 \) and $\dd\mu_\varphi=|\varphi^*|^2 \dd m$. 
\end{itemize}
Note that the function $\varphi$ is a multiplier of $\cD$, and that   \( \varphi \) is an extremal function for \( \mathcal{D} \), that is,
\[
\|\varphi\|_{\mathcal{D}} = 1, \qquad 
\langle \varphi, z^n \varphi \rangle_{\mathcal{D}} = 0, \quad n \ge 1.
\]
We refer to the monographs \cite{ARSW1,el2014primer} for numerous results on the Dirichlet space, and to the recent papers by Aleman, Hartz, McCarthy, and Richter, notably concerning Smirnov factorization in the Dirichlet space \cite{AHMR,AHMR1,AHRM2}.

The study of the Dirichlet space involves the study of the spaces $\cD (\mu)$. A particularly important case is \( \mathcal{D}(\delta_\zeta) \), where the norm is given by the \textit{local Dirichlet integral}.  In this context,} We establish new estimates for the local Dirichlet integral of a class of outer functions called distance functions. These functions are characterized  by boundary values \( |f^*(\zeta)| \) that depend only on \( {\dist}(\zeta, E) \), distance from $\zeta$ to $E$. More precisely, given a closed set \( E \subset \mathbb{T} \) of Lebesgue measure zero and a continuous function \( \omega : (0, \pi] \to \mathbb{R}^+ \) satisfying  
\[
\log \omega(\dist(\cdot, E)) \in L^1(\mathbb{T}),
\]
we define the outer function \( f_{\omega,E} \) by  
\begin{equation}\label{outerw}
| f_{\omega,E}^*(\zeta) | = \omega(\dist(\zeta, E)) \quad \text{a.e.}
\end{equation}
Carleson \cite{carleson1952} previously studied such functions and constructed outer functions that are smooth up to the boundary, with prescribed zero sets. Our main result provides an estimate for the local Dirichlet integral of these distance functions, with explicit bounds when \( E \) is an interpolation set for H\"older classes.

This leads us to investigate whether these functions belong to the algebra of multipliers of the Dirichlet space. Unlike in the Hardy or Bergman spaces, where multipliers are simply bounded holomorphic functions, the characterization of multipliers for the Dirichlet space is more intricate. Stegenga \cite{stegenga1980multipliers} characterized Carleson measures for the Dirichlet space in terms of the logarithmic capacity of their boundary values. Another characterization, due to Arcozzi, Rochberg, and Sawyer \cite{ARS1,ARS2,ARSW1}, links the reproducing kernel to the Carleson box. We use these results to identify examples of distance functions that belong to the algebra of multipliers.

Finally, we explore cyclicity in \( \cD(\mu) \). A central problem is to characterize cyclic functions in \( \cD(\mu) \), namely those whose polynomial multiples form a dense subset of \( \cD(\mu) \). Any cyclic function in \( \cD(\mu) \) must be outer and have a boundary zero set of $c_\mu$-capacity zero, where $c_\mu$ is the natural capacity associated with $\cD(\mu)$. The Brown-Shields conjecture \cite{BS} for \( \cD(\mu) \) suggests that this condition is also sufficient. While this conjecture has been extensively studied for the classical Dirichlet space (see, for example, \cite{EKR,EKR2,EKR3,EL,HS} and references therein), it remains open for general measures \( \mu \), even for continuous functions. 
Some progress has nevertheless been achieved: it was shown  in \cite{EEK} that the conjecture holds for measures with countable support, and in \cite{EEL}, new examples were given where it also holds. In this work, we contribute additional examples of measures and sets for which the conjecture is verified.

\subsection{Estimates of the local Dirichlet integral of distance functions}
Our main result provides estimates for the local Dirichlet integral of distance functions  $f_{\omega,E}$.  More precisely, let $E$ be a closed subset of the unit circle $\T=\partial\D$, $E$ is called a Carleson set if
$$
\int_\mathbb{T} \log \dist(\zeta,E)\, |\dd\zeta| > -\infty.
$$
A Carleson set is the zero set of a Hölder continuous analytic function. For $ \alpha > 0 $, we associate with such a set the following {\it distance function}, the outer function given by 
$$|f^*_{\alpha,E}(\zeta)| = \dist(\zeta,E)^\alpha \quad \text{a.e. on } \mathbb{T}.$$

Carleson \cite{carleson1952} was interested in the membership of the distance function in the Dirichlet space. He proved that $ f_{\alpha,E} \in \mathcal{D} $ for $ \alpha \geq 3 $. Let $ A^\infty(\mathbb{D}) = \text{\rm Hol}(\mathbb{D}) \cap \mathcal{C}^\infty(\mathbb{D}) $. Taylor and Williams \cite{TW} have shown that when $ E $ is a Carleson set, there exists an outer function $ f \in A^\infty(\mathbb{D}) $ such that $ f^{(n)}|_E = 0 $. It is clear that $ f \in \mathcal{M(D)} $, the algebra of multipliers of the Dirichlet space.
In \cite{EKR2}, the authors studied the membership of distance functions $ f_{\omega,E} $  in the Dirichlet space. Let  \(\omega\) be an increasing function such that \(\omega(t^\gamma)\) is concave  for some \(\gamma>2\), 
they showed by using Carleson formula \cite{carleson1960representation} that $ f_{\omega,E}\in \cD$ if  
\begin{equation}\label{normeDirichlet}
\int_\T \dist(\zeta,E)\omega'(\dist(\zeta,E))^2|d\zeta|<\infty, 
\end{equation}
In particular, if \( \alpha > \frac{1}{2} \), then \( f_{\alpha,E} \in \mathcal{D} \), and one can construct a Carleson set \( E \) such that \( f_{1/2,E} \notin \mathcal{D} \). In this paper, we investigate the membership of distance functions in the algebra of multipliers. To achieve minimal growth, we need to consider more regular Carleson sets.
Let \( I(\zeta,E) \) denote the connected component of \( \mathbb{T} \setminus E \) that contains \( \zeta \), and
\[
\Gamma (\zeta , E) := \left\{ \zeta' \in \mathbb{T} \setminus I(\zeta, E) : \dist(\zeta',E) \leq \dist(\zeta,E) \right\}.
\]
A closed set \( E \subset \mathbb{T} \) is an \( \mathcal{L}_1 \)-set (respectively an \( \mathcal{L}_2 \)-set) if for any \( \zeta \in \mathbb{T} \setminus E \), the following inequality holds:
\[
\tag{$\mathcal{L}_1$}
\int_{\Gamma (\zeta, E)} \frac{\log\left({\dist(\zeta,E)}/{\dist(\zeta',E)}\right)}{|\zeta - \zeta'|^2} \, |d\zeta'| \lesssim \frac{1}{\dist(\zeta,E)}, \quad \zeta \in \mathbb{T} \setminus E.
\]
\[
\tag{$\mathcal{L}_2$}
\int_{\Gamma (\zeta, E)} \frac{\log^2\left({\dist(\zeta,E)}/{\dist(\zeta',E)}\right)}{|\zeta - \zeta'|^2} \, |d\zeta'| \lesssim \frac{1}{\dist(\zeta,E)}, \quad \zeta \in \mathbb{T} \setminus E.
\]
Interpolation sets for H\"older classes and the generalized Cantor set are examples of such sets. 
These sets satisfy the following $\mathcal{K}$-condition (after Kotochigov).  A closed set \( E \subset \T \) is said to be a \(\mathcal{K}\)-set, if there exists a positive constant \( c_E \) such that for any arc \( I \subset \T \),
\begin{equation}\tag{$\mathcal{K}$}
\sup_{\zeta \in I} \dist(\zeta, E) \geq c_E |I|.
\end{equation}
Note that if $E$ is $\mathcal{K}$-sets, then for some $\beta >0$, $\dd(\cdot,E)^{-\beta}$  satisfies the Muckenhoupt’s condition $A_2$, and consequently  the $\mathcal{L}_1$ condition is fulfilled, \cite{B2}.

We now consider \( E \)  such that \( \log \omega({\dist}(\cdot, E)) \in L^1(\mathbb{T}) \). We associate the outer function \( f_{\omega, E} \) given by  \eqref{outerw}.  We obtain the following result, which provides a precise  estimate of the local Dirichlet integral of the distance function.\\

\begin{thm}\label{estimationlocal} Let \( \omega \) be an increasing function such that \( {\omega(t^\gamma)} \) is concave for some \( \gamma >2 \), and \( {t \omega'(t)}/{\omega(t)} \) is increasing.   Let $E$ be a Carleson set of $\T$. Then
 \begin{enumerate}
\item
Let $\alpha \in (0,1/2)$, then 
\begin{equation}\label{CK}
\cD _\zeta(f_{\alpha,E}) \asymp  \dist^{2\alpha -1}(\zeta, E) \text{   for all $\zeta\in \T$ } \iff E \text{ is an \(   \mathcal{L}_1 \)-set.}
\end{equation}
\item If $E$ is an \(   \mathcal{L}_2 \)-set then  
\begin{equation}\label{eqfomega}
\cD_\zeta(f_{\omega,E})\asymp \dist(\zeta,E)\omega'(\dist(\zeta,E))^2,\qquad \zeta\in \T
\end{equation}
\item If moreover $x\omega'(x) \asymp x^2\omega'(x^2)$, then 
\begin{equation}\label{eqfomega2}
\cD_\zeta(f_{\omega,E})\asymp \dist(\zeta,E)\omega'(\dist(\zeta,E))^2  \text{   for all $\zeta\in \T$ }  \iff E \text{ is an \(   \mathcal{L}_2 \)-set.}\\
\end{equation}
\end{enumerate}
\end{thm}

 To prove this result, we use the Richter-Sundberg formula \ref{RSfrormula} for outer functions in the Dirichlet space. Conditions $  \mathcal{L}$ are necessary to obtain this result. {  We may observe that the norm formula \eqref{normeDirichlet} can be recovered under the assumption that the additional set belongs to the class \( \mathcal{L} \). However, by considering the double integral and applying our estimates, we can in fact retrieve formula  \eqref{normeDirichlet} for any Carleson set, without requiring it to be in the class \( \mathcal{L} \).}

\subsection{Multipliers in the Dirichlet space}

An analytic function $ f $ defined on $ \mathbb{D} $ is a multiplier for $ \mathcal{D} $ if $ f \mathcal{D} \subset \mathcal{D} $. 
In other words, a function $ f \in \mathcal{D} $ is a multiplier of $ \mathcal{D} $ if 
$$\Vert fg\Vert_\mathcal{D} \lesssim \Vert g\Vert_\mathcal{D},\qquad g \in \mathcal{D}. $$
The algebra $ \mathcal{M}(\mathcal{D}) $ consisting of the multipliers of $ \mathcal{D} $ is not easy to describe.
 A direct reformulation in terms of Carleson measures for $ \mathcal{D} $ (that is, measures $\mu$  for which the embedding $\mathcal{D} \hookrightarrow L^2(\mathbb{D}, \mu)$ is continuous)  shows that $f$ is a multiplier for $\mathcal{D}$ if and only if $f$ is bounded and the measure
$d\mu(z) = |f’(z)|^2 \, dA(z) $ is a Carleson measure for $ \mathcal{D}$.
 
 A positive Borel measure $ \mu $ on $ \overline{\mathbb{D}} $ is called a Carleson measure for the Dirichlet space if 
 $$\int_{\overline{\mathbb{D}}} |g(z)|^2 \dd\mu(z) \lesssim \|g\|_{\mathcal{D}}^{2}\qquad g \in \mathcal{D}. $$
Carleson measures and multipliers  for $\mathcal{D}$ and  arise in various contexts, notably in the characterization of multipliers and interpolation sequences \cite{S}, as well as in the study of the corona theorem \cite{DKSTW,CSW,T}. This is also related to the weak product of the Dirichlet space \cite{AHRM2}.
Hartz \cite{H} recently generalized the result of \cite{RW}, showing that the multiplier algebra of the weak product coincides with that of the Dirichlet space, with equality of norms.  The question of identifying Carleson measures has been the subject of extensive research by numerous authors over the years.
 The characterization of Carleson measures given by Arcozzi, Rochberg, and Sawyer Theorem \ref{ARS}, see also \cite{ARSW1,ARSW2,ARS1,ARS2}, will allow us to give examples of distance functions that belong to  the algebra of multipliers. 
In \cite{Onebox}, 
there is  a characterization based solely on a "Carleson box," similar in nature to that used in the Hardy space.  We will study whether the distance function $f_{\omega,E}$ belongs to $\mathcal{M(D)}$. This will allow us to provide other concrete examples of multiplicative functions of $\cD$ with minimal growth, see [8], where further examples are also provided. Since $\mathcal{M}(\cD)$ is contained in $\cD$ and in the algebra of bounded holomorphic functions $H^\infty$, we are only interested in outer functions, since  the only inner functions that belong to the Dirichlet space are finite Blaschke products. 

We now state two results: the first concerns the specific function $f_{\alpha, E}$, while the second applies to the general case.

\begin{thm}\label{CarlesonMeasure3}
Let $\alpha \in (0,1/2)$ and  let \( E \in   \mathcal{L}_1 \). The following assertions are equivalent:
\begin{enumerate}
\item \( f_{\alpha,E} \in \mathcal{M}(\cD) \)
\item The measure \(\dd\mu_{\alpha,E}:= \dist(\zeta, E)^{2\alpha-1}  |\dd \zeta| \) is a Carleson measure for \( \cD \).
\item 
 For all arcs \( I \), the following inequality holds:
{\small \[
\int_{I} \int_{I} \log \frac{1}{|\zeta - \xi|} \,\dd\mu_{\alpha,E} (\zeta)\dd\mu_{\alpha,E}(\xi)
\lesssim  \mu_{\alpha,E} (I).
\]
}
\end{enumerate}
\end{thm}
\begin{thm}\label{CarlesonMeasure2}
Let \( \omega \) be an increasing function such that \( {\omega(t^\gamma)} \) is concave for some \( \gamma >2 \), and \( {t \omega'(t)}/{\omega(t)} \) is increasing.  Let  \( E \in   \mathcal{L}_2 \). The following assertions are equivalent:
\begin{enumerate}
\item \( f_{\omega,E} \in \mathcal{M}(\cD) \)
\item The measure \(\dd\mu_{\omega,E}:= \dist(\zeta, E) \omega'(\dist(\zeta, E))^2 \, |\dd \zeta| \) is a Carleson measure for \( \cD \).
\item
 For all arcs \( I \), the following inequality holds:
{\small \[
\int_{I} \int_{I} \log \frac{1}{|\zeta - \xi|} \,\dd\mu_{\omega,E} (\zeta)\dd\mu_{\omega,E}(\xi)
\lesssim  \mu_{\omega,E} (I).
\]
}
\end{enumerate}
\end{thm}

Recall that a function $f_{\omega,E}$  belongs to the Dirichlet space $\cD$ if and only if 
 $$ \int_{\T} \dist(\zeta, E)  \omega'(\dist(\zeta, E) )^2 \,\, |\dd \zeta| <\infty.$$
As an immediate consequence of Theorems 1.2 and 1.3, we obtain a simple geometric condition linking the set \( E \) with the function \( \omega \).
\begin{cor}
Let \( E \) be a closed subset of \( \T \), and let $\alpha \in (0,1/2)$. Under the same assumptions on \( \omega \) as above, we have:
\begin{enumerate}
\item If \( E \in \mathcal{L}_1 \), then \( f_{\alpha,E} \in \mathcal{M}(\mathcal{D}) \) whenever 
\[
\int_{\T} \frac{\log {1}/{\dist(\zeta, E)}}{\dist(\zeta, E)^{1-2\alpha}} \, |\dd \zeta| < \infty.
\]

\item If \( E \in \mathcal{L}_2 \), then \( f_{\omega,E} \in \mathcal{M}(\mathcal{D}) \) whenever 
\[
\int_{\T}  \dist(\zeta, E) \, \big(\omega'(\dist(\zeta,E))\big)^2 \log \frac{1}{\dist(\zeta, E)} \,\, |\dd \zeta| < \infty.
\]
\end{enumerate}
\end{cor}
Let \( (\zeta_n)_{n \geq 1} \) be a sequence of real numbers satisfying \( 0 < \zeta_n < 1/2 \) for all \( n \).  
Define  
\[
E_{(\zeta_n)} = \Big\{\exp\Big(2i\pi \sum_{n=1}^{\infty}\varepsilon_n \zeta_1\ldots\zeta_{n-1}(1 - \zeta_n), \quad \varepsilon_n =0\text{ or } 1\Big\}.
\]
The set $E_{(\zeta_n)}$ denotes the perfect set associated with the sequence $(\xi_n)$, for the construction of such a set, see \cite{KS}. When \( \zeta_n = \zeta \) for some fixed \( \zeta \in (0,1/2) \) and for all \( n \), we obtain the perfect symmetric set \( E_\zeta \) with constant ratio \( \zeta \).  
A notable example is the classical triadic Cantor set, which corresponds to \( \zeta = \frac{1}{3} \).
Recall that $\text{\rm dim}(E_\zeta)$, the Hausdorff dimension of $E_\xi$, is equal to $({\log 2})/{\log(1/\xi)}$.  Important examples of $\mathcal{L}$-sets are given by the generalized Cantor sets.  In fact, if   \( \limsup_n \zeta_n < 1/2 \), then $E_{(\zeta_n)}$ is an $\mathcal{L}_2$-set, see Lemma \ref{gkset} and  Proposition \ref{kholder}.  We use the one-box condition for Carleson measures in the Dirichlet space Theorem \ref{Onebox} (see \ref{Onebox} to derive multiplier conditions for the distance function \( f_{\omega,E} \) when \( E \) is a generalized Cantor set of positive logarithmic capacity. We show that, in this case, the distance function \( f_{\omega,E} \) is a multiplier if and only if it belongs to the Dirichlet space,  see Corollary \ref{corcantor}. In particular 
 {  We have 
the following result
 \begin{cor}  
Let $\xi \in (0, 1/2)$, and let $E = E_\xi$ be the perfect symmetric set of constant ratio $\xi$. Let $\alpha\in (0,1/2)$, the following assertions are equivalent:
\begin{enumerate}
\item $f_{\omega, E} \in \mathcal{D}$,
\item $f_{\omega, E} \in \mathcal{M}(\mathcal{D})$,
\item  $\displaystyle \int_{0}^1\omega'(t)^2 t^{1-\dim(E_\zeta)} \dd t <\infty.$
\end{enumerate}
\end{cor}
This result allows us to recover, in particular, the threshold for the inclusion of the power-distance outer functions $f_{\alpha,E_\xi}$ in the Dirichlet space and its multiplier algebra: 
$$f_{\alpha, E_\xi} \in \mathcal{D}\iff f_{\alpha, E_\xi} \in \mathcal{M}(\mathcal{D}) \iff \alpha > \frac{1}{2}\Dim(E_\xi).$$}

\subsection{Cyclicity in Dirichlet spaces}

We are also interested in the cyclicity in the Dirichlet space \( \cD(\mu) \). An important problem  is to identify functions \( f \in \cD(\mu) \) such that
\[
 \overline{\{ p f \ : \ p \text{ is a polynomial} \}} = \cD(\mu).
\]
These functions are called cyclic functions for $\mathcal{D}(\mu)$. When $\mu = 0$, the space 
$\mathcal{D}(\mu)$ reduces to the Hardy space $H^2$. Beurling’s theorem \cite{B} states that the cyclic vectors in 
$H^2$ are the outer functions. However, it is known that there exist outer functions in $\mathcal{D}(\mu)$ that are not cyclic in $\mathcal{D}(\mu)$. Indeed, the cyclicity of such a function depends on the distribution of the zeros of its radial limit  on the unit circle. When $\dd\mu(\zeta) =  \dd m(\zeta)/2\pi$, the normalized arc measure on $\T$, $\mathcal{D}(\mu)=\cD$, Recall that, as shown by Beurling \cite{el2014primer}, the radial limit of a function in the Dirichlet space exists quasi-everywhere. The Brown–Shields conjecture \cite{BS} asserts that $f \in \mathcal{D}$ is cyclic if and only if f is an outer function and the set of all zeros of the radial limit  $f^*$ has  logarithmic capacity zero.   An affirmative answer to this conjecture was given in \cite{EKR2,EKR3} for some Carleson sets in the context of the classical Dirichlet space, and in \cite{EEL,EEK,EEKAH,EKR,EL} for some $\cD(\mu)$.

It is important to note that if \( f \) is cyclic for \( \cD(\mu) \), then \( f \) is outer and \( c_\mu(\mathcal{Z}_\T(f)) = 0 \), where
\[
\mathcal{Z}_\T(f) = \left\{ \zeta \in \T : \inf_{r \to 1} |f(r\zeta)| = 0 \right\},
\]
and \( c_\mu \) is the capacity associated with \( \cD(\mu) \). The Brown-Shields conjecture \cite{BS} for \( \cD(\mu) \) asserts that the converse is also true.

Given $E \subset \mathbb{T}$ and $t \geq 0$, we define  
$$
E_t := \left\{\zeta \in \mathbb{T} : \dist(\zeta, E) \leq t \right\},$$
The following theorem is our main result. It extends \cite[Theorem 1.1]{EKR}, which was established for the classical Dirichlet space $\mathcal{D}.$

\begin{thm}
\label{thm:4.3}
Let \(f \in \mathcal{D}(\mu) \cap A(\D)\) be an outer function. Suppose that 
$
E = \mathcal{Z}_{\T}(f)
$
is a \(\mathcal{K}\)-set and 
$$
\mu(E_t) = O\bigl(h(t)\bigr), 
\quad t \to 0^+,
$$
where ${h(t)}/{t}\text{ is decreasing}$ and  \(h(t)/t^\sigma\) is   increasing for some $\sigma>0$. If 
\begin{equation}\label{integraldiversge}
\int_{0}^{2\pi} \frac{dt}{h(t)} = +\infty,
\end{equation}
then \(f\) is cyclic in \( \mathcal{D}(\mu)\).
\end{thm}
Clearly, \eqref{integraldiversge} implies  $c_\mu(E)=0$, since f is cyclic in $\mathcal{D}(\mu)$.
We show more generally  in Theorem \ref{polair2} that if the set E is an $\mathcal{L}_2$ set and \eqref{integraldiversge} is satisfied, then E is a polar set; equivalently, it has capacity zero.\\

{  The paper is organized as follows. 
In the next section, we present some elementary results on distance functions and provide a first estimate of the local Dirichlet integral of the distance function. 
In Section~3, we introduce the class \( \mathcal{L} \) and show that interpolation sets for Hölder spaces serve as examples. 
We also derive sharp estimates of the local Dirichlet integral for \( f_{\omega,E} \) and \( f_{\alpha,E} \), based on the Richter–Sundberg formula for outer functions in the Dirichlet space. We obtain an estimate of the local Dirichlet integral of the distance function \( f_{\omega,E} \) when \( \dist(\cdot, E)^{-\sigma} \in L^1(\T) \).   Section 4 is devoted to the membership of the distance function in $\mathcal{D}(\mu)$.
In Section~5, we apply Arcozzi’s theorem on Carleson measures to derive necessary and sufficient conditions for distance functions to be multipliers. 
Using our estimates of the local Dirichlet integral of distance functions, we obtain refined results in this framework. 
We also make use of the one-box condition for Carleson measures in the Dirichlet space to establish multiplier conditions for the distance function \( f_{\omega,E} \) when \( E \) is a generalized Cantor set of positive logarithmic capacity. 
Finally, in the last section, we first provide a sufficient condition for a closed subset of the unit circle to be polar. Then we investigate cyclicity in the space $\mathcal{D}(\mu)$, obtaining results that extend the classical case of the Dirichlet space.\\

\paragraph{A final word on notation.} We write \( A \lesssim B \) to mean that there exists a constant \( C > 0 \), independent of the relevant variables, such that \( A \leq C B \). Similarly, \( A \asymp B \) means that both \( A \lesssim B \) and \( B \lesssim A \) hold.

\section{\bf{Preliminary estimates of the local Dirichlet integral}}

Let $H^2$ be the Hardy space. Every function $f \in \mathcal{D}$ has non-tangential limits almost everywhere on $\mathbb{T}$. We denote by $f(\zeta)$ the non-tangential limit of $f$ at $\zeta \in \mathbb{T}$, if it exists. The function $f \in H^2$ is called an outer function if it is of the form
\[
f(z) = \exp\left(\frac{1}{2\pi} \int_\mathbb{T} \frac{\zeta + z}{\zeta - z} \log \varphi(\zeta) \, |\dd\zeta|\right), \qquad |z| < 1,
\]
where $\varphi$ is a nonnegative function in $L^1(\mathbb{T})$ such that $\log \varphi \in L^1(\mathbb{T})$. Note that $|f| = \varphi$ almost everywhere on $\mathbb{T}$.  The local Dirichlet integral of $f$  is  given by
\[
\mathcal{D}_\zeta(f) := \frac{1}{2\pi} \int_\mathbb{T} \frac{|f(\zeta) - f(\xi)|^2}{|\zeta - \xi|^2} |\dd\xi| < \infty.
\]

We recall the Richter–Sundberg formula \cite{RS}:
\begin{equation}\label{RSfrormula}
\mathcal{D}_{\zeta}(f) = \frac{1}{2\pi} \int_\mathbb{T} \frac{|f(\zeta')|^2 - |f(\zeta)|^2 - 2|f(\zeta)|^2 \log(|f(\zeta')|/|f(\zeta)|)}{|\zeta - \zeta'|^2} |\dd\zeta'|.
\end{equation}
Let $\zeta \in \mathbb{T} \setminus E$, and let $\Sigma$ be a Borel subset of $\mathbb{T}$. The local Dirichlet integral over $\Sigma$ is defined as 
$$
\cD_{\zeta, \Lambda}(f) = \int_{\Sigma} \frac{|f(\zeta')|^2 - |f(\zeta)|^2 - |f(\zeta)|^2 \log \left( {|f(\zeta')|^2}/{|f(\zeta)|^2} \right)}{|\zeta' - \zeta|^2} |\dd\zeta'|.
$$

Let $E$ be a closed subset of $\mathbb{T}$ of zero Lebesgue measure and let $\omega:\left[0,\pi\right]\rightarrow\mathbb{R}^+$ be a continuous function such that
\[\int_\mathbb{T}\log \omega(\dd(\zeta,E)) \vert \dd\zeta\vert>-\infty.\]
We associate the outer distance function
given by
$$f_{\omega,E}(z):=\exp\Big(\int_\T \frac{\zeta+z}{\zeta-z} \log \omega(\dist(\zeta,E)) \frac{|\dd\zeta|}{2\pi}\Big),\quad z\in \D.$$
In particular, 
$$|f^*_{\omega,E}(\zeta)|=\omega(\dd(\zeta,E))\quad \text{ a.e on } \T.$$

\subsection{First estimates of the local Dirichlet integral}

In order to estimate the local integral, we first decompose the circle into several subsets. The following sets will be used repeatedly throughout the paper:
\begin{align*}
I(\zeta, E) &:= \text{the connected component of } \mathbb{T} \setminus E \text{ containing } \zeta, \\
\Gamma(\zeta,E)& :=\Big\{\zeta'\in \mathbb{T} \setminus  I(\zeta, E) : \ \ \dd(\zeta', E) \leq \dd (\zeta, E)\Big\},\\
 \Sigma(\zeta,E)&:= \Big\{\zeta ' \in \T \setminus I(\zeta,E) : \ \ \dd (\zeta', E)  \geq \dd (\zeta, E) \Big\},\\
  \end{align*}
  
 \begin{lem}\label{lemma1} Let \( \omega \) be an increasing function such that \( {\omega(t^\gamma)} \) is concave for some \( \gamma >2 \), and \( {t \omega'(t)}/{\omega(t)} \) is increasing. Let  $E$ be a Carleson set of $\TT$ and let $\zeta \in \T \setminus E$. We have
\begin{align}
\cD _{\zeta, I(\zeta,E)}(f_{\omega,E}) &\lesssim \dist(\zeta,E)\omega'(\dist(\zeta,E))^2,\label{inI}\\
\cD _{\zeta, \Sigma(\zeta,E)}(f_{\omega,E}) &\lesssim\dist(\zeta,E)\omega'(\dist(\zeta,E))^2\label{inS},\\
\cD_\zeta(f_{\omega,E})&\gtrsim \dist(\zeta,E)\omega'(\dist(\zeta,E))^2\label{inD},
\end{align}
where the implied constants depend only on $\gamma$.
\end{lem}
\begin{proof}
Let $\zeta \in \T\setminus E$, set $\delta = d(\zeta, E)$ and $\delta' = d(\zeta', E)$.  Let $c$ be a large universal constant and write $$I:=I(\zeta,E)= I_1\cup I_2\cup
 I_3,$$
 where
 $$I_1= \{\zeta ' \in I:\ \ \delta ' < \delta/c\}, \; I_2= \{\zeta ' \in I:\ \ \delta/c  \leq \delta ' < c\delta \},\; 
I_3= \{\zeta ' \in I:\ \ \delta ' \geq  c \delta\}.$$
For $\zeta \in I_1$ we have 
\begin{equation}\label{eqforme}
\omega^2(\delta')-\omega^2(\delta)-2\omega^2(\delta)\log({\omega(\delta')}/{\omega(\delta)})= 4\int^{\delta}_{\delta'}\frac{\omega'(t)}{\omega(t)}\int^\delta_t\omega'(s)\omega(s)\dd s\, \dd t\\
\end{equation}
Then 
\begin{eqnarray}\label{eqI1}
\cD _{\zeta, I_1}(f_{\omega,E})&  \asymp  &   \int_{I_1} \left(\int^{\delta}_{\delta'}\frac{\omega'(t)}{\omega(t)}\int^\delta_t\omega'(s)\omega(s)\dd s\,\dd t\right)\ \frac{|\dd\zeta '|}{|\zeta-\zeta '|^2 }\nonumber\\
& \asymp  &\frac{1 }{\delta^{2}} \displaystyle \int_{I_1}\left( \int^{\delta}_{\delta'}\frac{\omega'(t)}{\omega(t)}\int^\delta_t\omega'(s)\omega(s)ds dt\right)|\dd\zeta '|\nonumber\\
&\asymp &\frac{1 }{\delta^{2}} \displaystyle \int_{r=0}^{\delta/c} \int^{\delta}_{t=r}\frac{\omega'(t)}{\omega(t)}\int^{\delta}_{s=t}\omega'(s)\omega(s)\dd s\,\dd t\, \dd r\nonumber\\
&\asymp &\frac{1 }{\delta^{2}} \displaystyle \int_{t=0}^{\delta} t\frac{\omega'(t)}{\omega(t)}\int^{\delta}_{s=t}s\omega'(s)\omega(s)\frac{\dd s}{s} \dd t\\
&\lesssim & \omega'(\delta)^2\displaystyle \int_{t=0}^{\delta}\log(\delta/t) \dd t \asymp \delta \omega'(\delta)^2\nonumber.
\end{eqnarray}
Let $\zeta \in I_2$, for $t\in (\delta',\delta)$ and $s\in (t,\delta)$, we have $\omega(t)\asymp\omega(s)\asymp\omega(\delta)$. Moreover, using \eqref{eqforme}, we obtain 
$$\omega^2(\delta')-\omega^2(\delta)-2\omega^2(\delta)\log({\omega(\delta')}/{\omega(\delta)})\asymp (\delta-\delta')^2\omega'(\delta)^2.$$
Hence
 \begin{eqnarray}\label{estimationsuperieur}
\cD _{\zeta, I_2}(f_{\omega,E})&  \asymp  & \omega'(\delta)^2 \displaystyle \int_{I_2}\frac{(\delta ' -\delta )^2}{|\zeta-\zeta '|^2 }|\dd\zeta '|
\lesssim  \omega'(\delta)^2\displaystyle \int_{I_2}|\dd\zeta '| \asymp \delta \omega'(\delta)^2.
\end{eqnarray}

Remark that for $\zeta ' \in I_3$ we have $|\zeta-\zeta '| \asymp \delta ' -\delta \asymp \delta '$. Let $u =|\zeta-\zeta'|$, then
\begin{eqnarray*}
\cD _{\zeta, I_3}(f_{\omega,E})& =& 4 \displaystyle \int _{\zeta ' \in I_3}\int_{t= \delta}^{\delta'}\frac{\omega'(t)}{\omega(t)}\int_{s=\delta}^t\omega'(s)\omega(s)
\dd s\, \dd t \frac{|\dd\zeta '|}{|\zeta-\zeta ' |^2}\\
&\lesssim & \displaystyle \int _{u= \delta}^1\int_{t= \delta}^{u}\frac{\omega'(t)}{\omega(t)}\int_{s=\delta}^t\omega'(s)\omega(s)\dd s\, \dd t \frac{\dd u}{ u^2}\\
&\asymp & \int_{s= \delta}^1\int_{t=s}^1 \frac{\omega'(t)}{\omega(t)}\omega'(s)\omega(s) \frac{\dd t}{ t} \dd s \\
&\asymp& \int_{s= \delta}^1\omega'(s)\omega(s)\int_{t=s}^1 t^{1-1/\gamma}\frac{\omega'(t)}{\omega(t)}\frac{\dd t}{ t^{2-1/\gamma}} \dd s \\
&\lesssim & \int_{s= \delta}^1\omega'(s)\omega(s)  s^{1-1/\gamma}\frac{\omega'(s)}{\omega(s)}\int_{t=s}^1 \frac{\dd t}{ t^{2-1/\gamma}} \dd s \\
&\asymp&  \int_{s= \delta}^1s^{2-2/\gamma}\omega'(s)^2 \frac{\dd s}{s^{2-2/\gamma}}\\
&\lesssim &\delta^{2-2/\gamma} \omega'(\delta)^2 \int_{s=\delta}^{1}\frac{\dd s}{s^{2-2/\gamma}}\dd s\asymp \delta \omega ' (\delta) ^2.
 \end{eqnarray*} 
  
 The first inequality of Lemma 8 follows from combining the three estimates. The inequality \eqref{inS} can be derived in exactly the same way as in the case of  $I_3$ since for $\zeta ' \in \Sigma (\zeta, E)$, we have  $\delta '   \lesssim |\zeta-\zeta '|$.   Finally, \eqref{inD} follows from  \eqref{estimationsuperieur}.  And this completes the proof of the lemma.
\end{proof}
We now turn our attention to the local integral \( \mathcal{D}_{\zeta, \Gamma(\zeta,E)}(f_{\omega,E}) \). Let us recall that
 $$ \Gamma(\zeta,E) :=\Big\{\zeta'\in \mathbb{T} \setminus  I(\zeta, E) : \ \ \dd(\zeta', E) \leq \dd (\zeta, E)\Big\}.$$

\begin{lem}\label{lemma2} Let \( \omega \) be an increasing function such that  \( {t \omega'(t)}/{\omega(t)} \) is increasing. Let $E$ be a Carleson set. Then 
$$\cD _{\zeta, \Gamma(\zeta,E)}(f_{\omega,E})\lesssim  \dist(\zeta,E)^2\omega'(\dist(\zeta,E))^2\int_{\Gamma(\zeta,E)} \frac{\log^2(\dist(\zeta,E)/\dist(\zeta',E))}{|\zeta-\zeta '|^2 }|\dd \zeta '|.$$
\end{lem}
\begin{proof} Let $\delta = d(\zeta, E)$ and $\delta' = d(\zeta', E)$.  
\begin{eqnarray}\label{eqgamma}
\cD _{\zeta, \Gamma(\zeta,E)}(f_{\omega,E})&=& 4
\int_{\Gamma(\zeta,E)  }  \left(\int^{\delta}_{t=\delta'}  \frac{\omega'(t)}{\omega(t)} \int_{s=t}^{\delta}  s\omega'(s)\omega(s)\frac{\dd s}{s} \dd t \right)\ \frac{
|\dd\zeta '|}{|\zeta-\zeta '|^2 }\\
& \lesssim  &\int_{\Gamma(\zeta,E)  }\delta {\omega'(\delta)}{\omega(\delta)}  \left( \int^{\delta}_{\delta'}t\frac{\omega'(t)}{\omega(t)}\log\big({\delta}/{t} \big) \frac{\dd t}{t} \right)\frac{|\dd \zeta '|}{|\zeta-\zeta '|^2 }\nonumber\\
& \lesssim  &\int_{\Gamma(\zeta,E)  }\delta^2 {\omega'(\delta)^2}  \left( \int^{\delta}_{\delta'} \log\big({\delta}/{t} \big) \frac{\dd t}{t}  \right)\frac{|\dd \zeta '|}{|\zeta-\zeta '|^2 }\nonumber\\
&\asymp&  \delta^2\omega'(\delta)^2 \int_{\Gamma(\zeta,E)  }  \frac{\log^2(\delta/\delta')}{|\zeta-\zeta '|^2 }|\dd\zeta '| \nonumber.
\end{eqnarray}
This concludes the proof.
\end{proof}
We shall restrict ourselves to the case  $\omega _\alpha (t)=t^\alpha$. 
 Given $\alpha>0$ and let $E$ be a Carleson set, we associate with such a set the  outer distance function given by 
$$
f_{\alpha, E}(z) = \exp\left (\alpha \displaystyle \int _\TT\frac{e^{it}+z}{e^{it}-z}\log \dist (e^{it},E) \frac{\dd t}{2\pi} \right), \qquad z\in \DD.
$$
In particular, 
$$|f^*_{\alpha,E}(\zeta)| = \dd(\zeta,E)^\alpha \quad \text{a.e. on } \mathbb{T}.$$
Note that $f_{\alpha,E}\in \cD$ if and only if
$d(\cdot,E)^{2\alpha-1}\in L^1(\T)$. The following result provides an improvement of Lemma \ref{lemma2} for $\omega_\alpha$.
\begin{lem}\label{lemma2bis} Let $E$ be a Carleson set and let $\alpha \in (0,1/2)$. Then 
$$\cD _{\zeta, \Gamma(\zeta,E)}(f_{\alpha,E})\lesssim  \dist(\zeta,E)^{2\alpha}\int_{\Gamma(\zeta,E)} \frac{\log (\dist(\zeta,E)/\dist(\zeta',E))}{|\zeta-\zeta '|^2 }|\dd \zeta '|.$$
\end{lem}
\begin{proof}
 Let $\delta = d(\zeta, E)$ and $\delta' = d(\zeta', E)$.  By \eqref{eqgamma}
\begin{eqnarray*}
\cD _{\zeta, \Gamma(\zeta,E)}(f_{\alpha,E})&\asymp & 
\int_{\Gamma(\zeta,E)  }  \left(\int^{\delta}_{t=\delta'}  \frac{1}{t} \int_{s=t}^{\delta}  s^{2\alpha -1}\dd s \dd t \right)\ \frac{
|\dd\zeta '|}{|\zeta-\zeta '|^2 }\\
& \lesssim  &\delta ^{2\alpha}\int_{\Gamma(\zeta,E)  }  \int^{\delta}_{t=\delta'}  \frac{\dd t}{t} \frac{
|\dd\zeta '|}{|\zeta-\zeta '|^2 }\\
&\asymp&  \delta^{2\alpha} \int_{\Gamma(\zeta,E)  }  \frac{\log (\delta/\delta')}{|\zeta-\zeta '|^2 }|\dd\zeta '| \nonumber.
\end{eqnarray*}
This concludes the proof.

\end{proof}


\section{\bf{Minimal growth and $\mathcal{L}$-sets }}
 Our goal in this subsection is to characterize the class of subsets $E$ of $\TT$ having the minimal growth. By  Lemma \ref{lemma1}, this means that  
$$\cD_\zeta (f_{\omega,E}) \asymp \dist(\zeta,E) \omega'(\dist(\zeta,E))^2,\quad a.e. \ \zeta \in \TT.$$
To achieve minimal growth, we need to consider more regular Carleson sets. A closed set $E$ on $\T$ is said to be a $\mathcal{K}$-set  (after Kotochigov), we will write $E\in \cal{K}$, if there exists a positive constant $c_E$ such that for any arc $I \subset \T$,
\begin{equation}\tag{$\mathcal{K}$}
\sup_{\zeta\in I} \dist(\zeta, E) \geq c_E |I|,
\end{equation}
where $|I|$ denotes the length of $I$.  In other words, any arc contains a point whose distance to $E$ is on the order of the length of the arc. $\mathcal{K}$-sets appear as interpolation sets for Hölder classes \cite{D}, see also \cite{B1,B2}. Note that for every $\mathcal{K}$--set  $E$ there exists $\beta \in (0,1)$ such thatt  $\dist(\zeta,E)^{-\beta}$ is a $(A_2)$-Muckenhoupt weight, we refer to \cite{B1,B2,D,EKR} for more details. The following lemma gives us some metric  properties of $\mathcal{K}$-sets.

 \begin{lem}\label{kset} Let $E$ be a closed subset of $\T$. Then the following conditions are equivalent
\begin{enumerate}
\item $E\in \cal{K}$.
\item  There exists  $c>0$ such that for an arbitrary arc $I$  
$$ \displaystyle\frac{1}{|I|} \int_I \log (|I|/{\dist(\zeta,E)}) |\dd\zeta| \leq c.$$
\item  There exists  $c>0$ such that, for an arbitrary arc $I$  and for some $0<\beta<1$, we have 
$$
\frac{1}{|I|}\int_I {\dist(\zeta,E)^{-\beta}}|\dd\zeta| \le c |I|^{-\beta}
$$
\item There exists  $c>0$ such that, for an arbitrary arc $I$  and for some $0<\beta<1$, we have 
$$
\frac{1}{|I|}\int_I {\dist(\zeta,E)^{\beta}}|\dd\zeta| \geq c |I|^{\beta}
$$
\item  There exists  $\beta_E\in (0,1]$, such that for any $\beta \in (0,\beta_E)$ there exists $c>0$ and  for any arbitrary arc $I$,    we have 
$$ |\{\zeta \in I : d(\zeta,E) \le t\}| \leq c |I| (t/|I|)^{\beta_E}. $$
 \end{enumerate}
\end{lem}
All these conditions imply that \( \log \dist(\cdot,E) \in \mathrm{BMO} \) and that \(  \dist(\cdot,E)^{-\beta} \) satisfies the Muckenhoupt's condition $A_2$.
Dyn'kin in \cite{D}  showed that the value of $\beta$ in (3) is required to satisfy
$$\beta < \big({\log ({1}/{1-c_E}})\Big)\Big/ \Big({\log ({2}/{1-c_E}})\big).$$

Note that $E$ has measure zero and $\log \dist(\cdot, E) \in L^1(\T)$, thus $E$ is a Carleson set,  An important example  of a  $\mathcal{K}$-sets are the generalized Cantor sets,  
 see Proposition \ref{kholder}. Now we will introduce a class of subsets of $\TT$ which seems close to the class of $\mathcal{K}$-sets.

Recall that the connected component of $\mathbb{T} \setminus E$ containing $\zeta$ will be denoted by $I(\zeta, E)$  and let 
$$
\Gamma (\zeta , E):= \big\{\zeta' \in \T \setminus I(\zeta, E) :  \dist(\zeta ',E)\leq \dist(\zeta,E) \big\} 
$$
We next define the notion of $\mathcal{L}_1$ and $\mathcal{L}_2$-sets.
\begin{defi}
Let $E \subset \TT$ be a closed Carleson set.  
We say that $E$ belongs to the class $\mathcal{L}_1$ (respectively $\mathcal{L}_2$) if there exists a constant $C > 0$ such that, for every $\zeta \in \TT \setminus E$, the following inequalities hold:
 \begin{equation} \tag{$\mathcal{L}_1$}\displaystyle \int_{\Gamma (\zeta, E)} \frac{\log(\dist(\zeta,E))/\dist(\zeta',E))}{|\zeta- \zeta'|^2} |\dd\zeta' |
 \leq \frac{C}{\dist(\zeta,E)},\qquad \zeta \in \T \setminus E.
 \end{equation}
 \begin{equation} \tag{$\mathcal{L}_2$}\displaystyle \int_{\Gamma (\zeta, E)} \frac{\log^2(\dist(\zeta,E))/\dist(\zeta',E))}{|\zeta- \zeta'|^2} |\dd\zeta' |
 \leq \frac{C}{\dist(\zeta,E)},\qquad \zeta \in \T \setminus E.
 \end{equation}
The best constant $C$ in the above inequalities will be denoted by $C_E$.
\end{defi}

If $E$ is a $\mathcal{K}$-set,  \(  \dist(\cdot,E)^{-\beta} \) satisfies the Muckenhoupt's condition $A_2$,  the \( \mathcal{L}_1 \) condition is thereby fulfilled (indeed \( A_2 \) implies \( B_2 \), see \cite{B2}). The following result provides insight into the relationship between the class of $\mathcal{K}$-sets and the class $  \mathcal{L}_2$.
 \begin{lem}\label{gkset}  We have the following assertions
 \begin{enumerate}
\item $\cal{K} \subset   \mathcal{L}_2$.
\item  Let $E$ be a Carleson set. If for every closed subset $F\subset E$ we have $F \in   \mathcal{L}_2$ with $\sup \{C_F: \  F\subset E \} <\infty$ then $E \in \cal{K}$. 
 \end{enumerate}
 \end{lem}
\begin{proof}
(1) Let $E \in \cal{K}$, by Lemma \ref{kset} (1), (4) and Jensen inequality we have 
$$ \displaystyle\frac{1}{|I|} \int_I \log^2 (|I|/{\dist(\zeta,E)}) |\dd\zeta| \leq c.$$
Put $\delta = \dist(\zeta,E)$, $\delta '= \dist(\zeta ',E)$.
{  \begin{eqnarray*}
\displaystyle \int_{\Gamma (\zeta, E)} \frac{\log^2(\delta/\delta ' )}{|\zeta- \zeta'|^2} |\dd\zeta' | & \leq & \displaystyle \sum _{n\geq 1} \int _{2^n \delta \leq |\zeta '-\zeta |\leq 2^{n+1}\delta} \frac{\log^2(\delta/\delta ')}{|\zeta- \zeta'|^2} |\dd\zeta' | \\
& \lesssim &  \displaystyle \sum _{n\geq 1}\frac{1}{2^{2n}\delta ^2} \int _{ |\zeta '-\zeta |\leq 2^{n+1}\delta} \log^2(2^{n+1}\delta/\delta ')|\dd\zeta' | \\
&\lesssim &  \displaystyle \sum _{n\geq 1}\frac{1}{2^{n}\delta } \asymp 1/\delta.
\end{eqnarray*}}
Then $E \in   \mathcal{L}$.\\

To prove the second assertion, let $E\in   \mathcal{L}_2$ and let $I$ be a closed arc.  Without loss of generality, we may assume that $E\cap I \neq \emptyset$. Let $\zeta \in \TT$ such that $\dist (\zeta, I) = |I|/2$ and let $J=\{ \zeta ' \in \TT: \ |\zeta-\zeta '|\leq |I|/2\}$. By assumption $E\cap I \in   \mathcal{L}_2$ then
\begin{eqnarray*}
{ \Big(\frac{1}{|I|} \int_I \log (|I|/{\dist(\zeta ',E)}) |\dd\zeta '|\Big)^2}& \lesssim & 
 \frac{1}{|I|} \int_I \log^2 (|I|/{\dist(\zeta ',E)}) |\dd\zeta '|\\
 & \lesssim &  \frac{1}{|I|} \int_{I\cup J}\log^2 (\dist(\zeta, E\cap I)/{\dist(\zeta ',E\cap I)}) |\dd\zeta '|\\
& \lesssim &|I|   \int_{I\cup J}\frac{ \log^2 (\dist(\zeta, E\cap I)/{\dist(\zeta ',E\cap I)})}{|\zeta-\zeta '|^2} |\dd\zeta '|\\
 & \lesssim & C_{E\cap I},\\
\end{eqnarray*}
then $E \in \cal{K}$.
\end{proof}

\subsection{Remarks and examples}\label{remarks}

{\bf 1.}  The class $  \mathcal{L}_2$ is strictly larger than $\cal{K}$. Indeed, let $\gamma >0$, $a_n=n^{-\gamma}$,  and let $E$ be the closed set given by
$$E= \{e^{\pm ia_n}: n\geq 1\}\cup \{1\}.$$ 
 One can verify easily that $E $ is not a $\mathcal{K}$-set. To verify that $E\in   \mathcal{L}_2$,  we have to prove that 
 $$\int_{\Gamma (\zeta, E)} \frac{\log^2(\dist(\zeta,E))/\dist(\zeta',E))}{|\zeta- \zeta'|^2} |\dd\zeta' |=O(1/\dist(\zeta,E)).$$
 Let $\zeta = e^{i\theta}$ with $\theta \in (a_{N+1}, a_{N})$ for some integer $N$.  
 Note that $I(\zeta,E)=I_N$, set $\delta=\dist(\zeta,E)$ and $\delta'=\dist(\zeta',E)$, hence the most relevant part of this integral is 
 $$ \int_{ \displaystyle \cup_{n\neq  N} I_{n}; \delta ' <\delta} \frac{\log^2(\dist(\zeta,E))/\dist(\zeta',E))}{|\zeta- \zeta'|^2} |\dd\zeta' |.$$
Note that $\int_{I_k}\log^2(\delta/\delta')\lesssim |I_k|\log^2(\delta/|I_k|)$. Then we have 
\[
\begin{array}{lll}
\displaystyle \int_{ \displaystyle \cup_{n\geq  2N} I_{n}} \frac{\log^2(\dist(\zeta,E))/\dist(\zeta',E))}{|\zeta- \zeta'|^2} |\dd\zeta' |& \lesssim &\displaystyle \sum_{k\geq 2N}\frac{1}{a_N^2}\int_{I_k}\log^2(\delta/\delta')|d\zeta'|\\
 & \lesssim &N^{2\gamma}\displaystyle \sum_{k\geq 2N} \frac{1}{k^{1+\gamma }}\log (\delta k^{1+\gamma})\\
 & \lesssim & N^{2\gamma}\delta^{\gamma/1+\gamma}\displaystyle \int _{2N}^\infty \frac{\log ^2x}{x^{\frac{1+2\gamma}{1+\gamma}}}dx\\
\\
 & \lesssim & N^{2\gamma}\delta^{\gamma/1+\gamma}\lesssim 1/\delta ^{\gamma/1+\gamma} \lesssim 1/\delta.\\
 \end{array}
 \]

If $N+2\leq k\leq 2N$,  $a_k-a_{k+1} \asymp N^{-\gamma-1}$ and $a_N-a_k\asymp (k-n)N^{-\gamma-1}$. Then
\[
\begin{array}{lll}
\displaystyle \int_{ \displaystyle \cup_{n=N+2}^{2N} I_{n}; \delta' <\delta} \frac{\log^2(\dist(\zeta,E))/\dist(\zeta',E))}{|\zeta- \zeta'|^2} |\dd\zeta' |& \lesssim &\displaystyle \sum_{k=N+2}^{2N}\frac{1}{(a_N-a_k)^2}\int_{I_k; \delta'<\delta} \log^2(\delta/\delta')|d\zeta'|\\
& \lesssim &\displaystyle \sum_{k=N+2}^{2N}\frac{N^{2\gamma +2}}{(k-N)^2}\int_{I_{k};\delta'<\delta} \log^2(\delta/\delta')|d\zeta'|\\
& \lesssim &\displaystyle \sum_{k=N+2}^{2N}\frac{N^{2\gamma +2}}{(k-N)^2}\int_{0}^\delta \log^2(\delta/t)|dt\\
& \asymp & \delta N^{2\gamma+2}\leq 1/\delta.
\end{array}
\]
With the same arguments, one can treat the integrals corresponding to the remaining parts and conclude that $E\in \cal{L}_2$.

{\bf 2. } One can verify that the set 
$$F= \{e^{- in^{-\gamma}}: n\geq 1\}\cup \{1\},$$
does not belong to $  \mathcal{L}_2$. Indeed, using the same notation as before, one can prove easily  that
$$
\displaystyle \int_{ \cup_{n\geq N}I_n} \frac{\log^2(\dist({\zeta},F))/\dist(\zeta',F))}{|\zeta- \zeta'|^2} |\dd\zeta' | \asymp \frac{\log^2 1/\dist({\zeta},F)}{\dist(\bar{\zeta},F)}.
$$

{\bf 3.}  { An important examples of a $\mathcal{K}$-sets and hence $  \mathcal{L}_2$-sets are the generalized Cantor sets. 
The following result is simple and likely well-known

\begin{prop}\label{kholder}
Let \(E\) be the Cantor set associated with the sequence \((\zeta_n)_{n \ge 1}\).  
If 
\[
\limsup_{n \to \infty} \zeta_n < \frac{1}{2},
\]
then \(E\) satisfies the \(\mathcal{K}\)-condition.
\end{prop}
\begin{proof}
We want to show that \( \sup_{\zeta \in L} \mathrm{d}(\zeta, E) > c|L| \). Since the union of finitely many \( \mathcal{K} \)-sets is again a \( \mathcal{K} \)-set, we may assume \( \zeta_n < \eta < \tfrac{1}{2} \) for $n\geq 1$ .  It is also easy to see that we may reduce to the case where \( L = I_{k,n} \) for some \( n \geq 1 \), \( 1 \leq k \leq 2^n \), and that 
$
\sup_{\zeta \in I_{k,n}} \mathrm{d}(\zeta, E) \asymp \mathrm{d}(\widetilde{\zeta}, E),
$
where \( \widetilde{\zeta} \) is the midpoint of \( I_{k,n} \). To conclude, we observe that
\[
\mathrm{d}(\widetilde{\zeta}, E) \geq  \pi (\zeta_1\ldots\zeta_n)(1 - 2\zeta_{n+1}) 
\geq \frac{(1 - 2\eta)}{2}{|I_{k,n}|}. \qedhere
\]
\end{proof}
}
\subsection{Proof of Theorem \ref{estimationlocal}}

Theorem \ref{estimationlocal} follows from Lemma \ref{estimationgamma}, Lemma  \ref{lem9} and Lemma \ref{lem11}.

\begin{lem}\label{estimationgamma}  Let \( \omega \) be an increasing function such that \( {\omega(t^\gamma)} \) is concave for some \( \gamma >2 \), and \( {t \omega'(t)}/{\omega(t)} \) is increasing. Let $E\in   \mathcal{L}_2$ then  
\begin{equation}\label{eqfomega}
\cD_\zeta(f_\omega)\asymp\dist(\zeta,E)\omega'(\dist(\zeta,E))^2,
\end{equation}
where the implied constants depend only on  $\gamma$.
\end{lem}
\begin{proof} Let $\delta = d(\zeta, E)$ and $\delta' = d(\zeta', E)$.  Note that $\TT= I(\zeta,E)\cup \Sigma(\zeta,E)\cup \Gamma (\zeta,E).$
We have
$$
\cD _{\zeta}(f_{\omega,E})= \cD _{\zeta, I(\zeta,E)}(f_{\omega,E})+\cD _{\zeta, \Sigma(\zeta,E)}(f_{\omega,E})+\cD _{\zeta, \Gamma(\zeta,E)}(f_{\omega,E})
$$
\noindent  By Lemma \ref{lemma1}, we have $$
 \cD _{\zeta, I(\zeta,E)}(f_{\omega,E})+\cD _{\zeta, \Sigma(\zeta,E)}(f_{\omega,E}) \lesssim \dist(\zeta,E)\omega'(\dist(\zeta,E))^2.
 $$
\noindent  We also have by Lemma \ref{lemma2}

$$  \cD _{\zeta, \Gamma (\zeta,E)}(f_{\omega,E}) \lesssim  \dist(\zeta,E)^2\omega'(\dist(\zeta,E))^2\int_{\Gamma(\zeta,E)} \frac{\log^2(\dist(\zeta,E)/\dist(\zeta',E))}{|\zeta-\zeta '|^2 }|\dd \zeta '|.$$
\noindent  Since the set $E\in \mathcal{L}_2$, we obtain
$$  \cD _{\zeta, \Gamma (\zeta,E)}(f_{\omega,E}) \lesssim  \dist(\zeta,E)\omega'(\dist(\zeta,E))^2$$
Combining these inequalities,  we obtain \eqref{eqfomega}.
\end{proof}


\begin{lem}\label{lem9}
Let $\alpha \in (0,1/2)$.  Let $E$ be a Carleson set of $\T$. Then 
\begin{equation}\label{CK1}
\cD _\zeta(f_{\alpha,E}) \asymp \dist^{2\alpha -1}(\zeta, E),\qquad \zeta \in \T \setminus E
\end{equation}
if and only if $E\in   \mathcal{L}_1$.
\end{lem}
 {\begin{proof} Let $\delta = d(\zeta, E)$ and $\delta' = d(\zeta', E)$.  Suppose that $E\in   \mathcal{L}_1$. By Lemma \ref{lemma1} and Lemma \ref{lemma2bis}, we have
 $$
 \cD _\zeta(f_{\alpha,E}) \lesssim \delta^{2\alpha }\displaystyle \int _{\Gamma(\zeta,E)}\frac{\log(\delta/\delta')}{|\zeta-\zeta'|^2}|\dd\zeta'|\lesssim \delta ^{2\alpha-1}.
 $$
The reverse inequality is proved in Lemma \ref{lemma1}.\\
 
 \noindent Conversely, suppose that (\ref{CK1}) holds. 
 Let $c>1$ be such that $2\alpha \log c  >1$. Write $\Gamma(\zeta,E)=\Gamma_1\cup\Gamma_2$ where 
 $$\Gamma_1=\{\zeta'\in \Gamma(\zeta,E)\text{ : } \delta'\leq \delta/c\}, \quad \text{ and } \quad  
 \Gamma_2=\{\zeta'\in \Gamma(\zeta,E)\text{ : } \delta/c \leq \delta'\leq \delta\}.$$
 Note that for $\zeta' \in \Gamma _1(\zeta, E)$, we have 
 $$
 \delta'^{2\alpha}-\delta^{2\alpha}- 2 \delta^{2\alpha}\log( \delta'^{\alpha}/ \delta^{\alpha}) \asymp  \delta^{2\alpha}log( \delta /\delta')
 $$
 Then we obtain
  $$\cD_{\zeta,\Gamma_1}(f_{\alpha,E})\asymp \delta^{2\alpha}\int_{\Gamma_1} \frac{\log(\delta/\delta')}{|\zeta-\zeta'|^2}|\dd\zeta'|, $$
  We also have 
  $$
\int_{\Gamma_2} \frac{\log(\delta/\delta')}{|\zeta-\zeta'|^2}|\dd\zeta'|\lesssim1/\delta 
  $$
The result comes from these two estimates combining with Lemma \ref{lemma1} and Lemma \ref{lemma2bis}.
 \end{proof}}
\begin{lem}\label{lem11}
Let $E$ be a Carleson set of $\T$ such that $|E_t|= O(t^\varepsilon )$ for some $\varepsilon \in (0,1)$. Let \( \omega \) be an increasing function such that \( {\omega(t^\gamma)} \) is concave for some \( \gamma >2 \), and \( {t \omega'(t)}/{\omega(t)} \) is increasing. Further assume that  $x\omega(x)\asymp x^2\omega'(x^2)$ then 
\begin{equation}\label{CK}
\cD _\zeta(f_{\omega,E}) \asymp  \dist(\zeta,E)\omega'(\dist(\zeta,E))^2,\qquad \zeta \in \T \setminus E
\end{equation}
if and only if $E\in   \mathcal{L}_2$.
\end{lem}
 {\begin{proof} Let $\delta = d(\zeta, E)$ and $\delta' = d(\zeta', E)$.  
 By Lemma \ref{estimationgamma}, it is sufficient to prove the necessary condition.
Let  $\beta>0$ be a large constant and set 
$$\Gamma_1=\{\zeta' \in \TT\text { : } |\zeta -\zeta'| >\delta \text{ and } \delta^\beta <\delta' \leq \delta \} \quad \text{ and } \quad  \Gamma_2 :=\{\zeta' \in \TT\text { : } |\zeta -\zeta'| >\delta \text{ and }  \delta' \leq \delta^\beta  \}. $$ 
Let 
\[
\Delta(\delta,\delta') = \omega^2(\delta') - \omega^2(\delta) - 2\omega^2(\delta') \log \frac{\omega(\delta')}{\omega(\delta)}
\]

\noindent For $\delta^{\beta} < \delta' \leq \delta$, we have
$$
\delta^\beta\omega'(\delta^\beta)\asymp \delta\omega'(\delta) \quad  \text {and } \quad \omega(t) \asymp \omega(s), \quad \delta' \leq t \leq s\leq  \delta.
$$
Thus we obtain
\begin{eqnarray*}
\Delta(\delta,\delta') &=& \int_{t=\delta'}^{\delta} \int^{\delta}_{s=t} \frac{\omega'(t)}{\omega(t)}\omega'(s)\,\omega(s)\, ds \, dt\\
&\asymp& \int_{\delta'}^{\delta} \int_{\delta}^{t} \omega'(t)\,\omega'(s)\, ds\, dt\\
&=& \frac{1}{2} \bigl(\omega(\delta) - \omega(\delta')\bigr)^2 \\
&=&\frac{1}{2}\Big(\int_{\delta'}^\delta  x\omega'(x) \frac{dx}{x}\Big)^2\\
&\geq& (\delta'\omega'(\delta') )^2\log^2(\delta/\delta')\asymp  (\delta \omega'(\delta))^2 \log^2(\delta/\delta')
\end{eqnarray*}
Hence 
$$\cD_{\zeta,\Gamma_1}(f_{\omega,E}) = \int_{\Gamma_1}\frac{\Delta(\delta,\delta')}{|\zeta-\zeta'|^2}|\dd \zeta'| \geq \delta^2 (\omega'(\delta))^2 \int_{\Gamma_1}\frac{\log^2(\delta/\delta')}{|\zeta-\zeta'|^2}|\dd \zeta'|$$
This implies the following inequality
$$
\displaystyle \int _{\Gamma_1}\frac{ \log^2 (\delta /\delta ')}{|\zeta-\zeta'|^2}|\dd \zeta'| \lesssim \frac{1}{\delta }.
$$
On the other hand, since \( |E_t| = O(t^\varepsilon) \), $\log^4( 1 /\delta') \in L^1(\TT)$.  Hence for $\beta \varepsilon >4$ , we obtain
\begin{eqnarray*}
 \int_{\Gamma_2}
\frac{\log^2( \delta /\delta')}{|\zeta-\zeta'|^2} \, |\dd\zeta' |& \lesssim & \frac{1}{\delta^2} \int_{\Gamma_2}
\log^2( 1 /\delta')\, |\dd\zeta' |\\
& \lesssim &   \frac{1}{\delta^2}  \left ( \int_{\Gamma_2} |\dd\zeta' |\right )^{1/2}\left (\int_{\Gamma_2}
\log^4( 1 /\delta')\, |\dd\zeta' |\right )^{1/2} \\
&\lesssim  &  \frac{1}{\delta^2}|E_{\delta^\beta}|^{1/2}\left (\int_{\Gamma_2}
\log^4( 1 /\delta')\, |\dd\zeta' |\right )^{1/2}.\\
&\lesssim  & \delta^{\frac{\beta \varepsilon}{2}-2}\left (\int_{\TT}
\log^4( 1 /\delta')\, |\dd\zeta' |\right )^{1/2} \lesssim \frac{1}{\delta }.\\
\end{eqnarray*}
Therefore $E\in  \mathcal{L}_2$, which completes the proof.
\end{proof}}


To conclude this section, we examine a situation where the condition $E \in \mathcal{L}_2$ is not necessary satisfied.
We have the following result.

\begin{thm}\label{GammaCarleson}
Let \( \omega \) be an increasing function such that \( {\omega(t^\gamma)} \) is concave for some \( \gamma >2 \), and \( {t \omega'(t)}/{\omega(t)} \) is increasing. Let  $E\subset \T$ be  be a closed set  such that    $\dist(\cdot, E)^{-\sigma} \in L^1(\T)$ for some $\sigma\in (0,1)$. Then 
 $$
\cD _\zeta (f_{\omega,E}) \lesssim \dist(\zeta,E)\omega'(\dist(\zeta,E))^2\log^2(1/\dist(\zeta,E)) 
 $$
 { Conversely, there exists a set $E$ such that $\dist(\cdot, E)^{-\sigma}\in L^1(\T)$ for some $\sigma>0$ and 
 $$
\displaystyle \sup _{d(\zeta,E)\geq \delta}\cD _\zeta (f_{\omega,E}) \asymp \delta \omega'(\delta)^2\log ^2(1/\delta) , \qquad \delta \in (0,1).
 $$
}
 \end{thm}
\begin{proof}
Using the same notation as in the previous proof, and by Lemma \ref{lemma1} and Lemma \ref{lemma2}, It suffices to prove 
$$
 \int_{\dist(\zeta, \zeta') \geq \delta} \frac{\log^2 ({\delta}/{\delta'})}{\dist(\zeta, \zeta') ^2} |\dd\zeta' | \lesssim \frac{1}{\delta}\log ^2(1/\delta).
$$
Let $p,q, \varepsilon $ such that $$\frac{1}{p}+\frac{1}{q}=1, \quad q=\log(1/\delta),\quad \text{ and}\quad \varepsilon =\frac{\sigma}{2q}.$$
By using Hölder’s inequality,  together with the inequality $\log x\lesssim \frac{x^\varepsilon}{\varepsilon}$,  we obtain
{ \begin{align*}\label{estgcarleson}
 \int_{\dist(\zeta, \zeta') \geq \delta} \frac{\log^2 ({\delta}/{\delta'})}{\dist(\zeta, \zeta') ^2} |\dd\zeta' |  
&\lesssim  \Big(\int_{\T} \Big(\log^2 \big( \frac{1}{\delta'}\big)\Big)^q |\dd\zeta' |\Big)^{1/q} \Big( \int_{\dist(\zeta, \zeta') \geq \delta} \frac{ |\dd\zeta' | }{\dist(\zeta, \zeta') ^{2p}}  \Big)^{1/p}\nonumber\\
&\lesssim \frac{1}{\varepsilon^2}  \Big(\int_{\T}  \frac{|\dd\zeta' |}{{\delta' }^{\sigma }}\Big)^{1/q} \Big(\frac{1}{(2p-1)\delta^{2p-1}}\Big)^{1/p}\nonumber\\
&\lesssim \frac{ 1}{\varepsilon^2 \delta}\\
&\lesssim \frac{ 1}{ \delta}\log^2(1/\delta).
\end{align*}
}
This completes the proof of the desired result.
\end{proof}

\subsection*{Remark} The second example in Remarks \ref{remarks} shows the optimality of the condition; it is enough to take $(1+\gamma)(1-\sigma) > 1$.

\section{\bf {Membership  to $\cD(\mu)$}}
The first application of the local Dirichlet integral is the membership to $\cD (\mu)$. We have the following result.
\begin{thm}\label{appartdmu}
 Let $E$ be a Carleson set of $\T$ and let $\mu$ be a positive measure on $\TT$.
 \begin{enumerate}
\item Let $\alpha \in (0,1/2)$ and suppose that $E\in \mathcal{L}_1$, then
\begin{equation*}
f_{\alpha,E} \in \cD(\mu) \iff \displaystyle \int _\TT \dist(\zeta,E)^ {2\alpha -1}\dd \mu (\zeta)<\infty.
\end{equation*}
\item Let \( \omega \) be an increasing function such that \( {\omega(t^\gamma)} \) is concave for some \( \gamma >2 \), and \( {t \omega'(t)}/{\omega(t)} \) is increasing. Suppose that $E \in   \mathcal{L}_2 $, then  
\begin{equation*}
f_{\omega,E} \in \cD(\mu) \iff \displaystyle \int _\TT \dist(\zeta,E)\omega'(\dist(\zeta,E))^2\dd \mu (\zeta)<\infty.
\end{equation*}
\end{enumerate}
\end{thm}
\begin{proof}
It is a direct consequence of Theorem \ref{estimationlocal}.
\end{proof}
It is worth noting that, in the classical Dirichlet space $\cD$, the same result was proved in \cite{EKR2} without the additional $\mathcal{L}$ condition. In the following theorem we give a class of measures for which the $\mathcal{L}$ condition can be removed.
\begin{thm}
Let \( \mu \) be the measure defined by \( \dd\mu(\zeta) = \dist(\zeta, E)^\alpha |\dd\zeta| \) with \( \alpha \geq 0 \), and let \( E \subset \mathbb{T} \) be a Carleson set. Assume that $\omega(t^\gamma)$ 
is concave for some \( \gamma > {2}/{(2 - \alpha)} \) and $t\omega'(t)/\omega(t)$ is is increasing. Then, we have 
\[
 \mathcal{D}_\mu(f_{\omega, E}) \asymp \int_\mathbb{T} \dist(\zeta, E)\, \omega'(\dist(\zeta, E))^2 \, \dd\mu(\zeta).
\]
\end{thm}

\begin{proof}
Let \( \zeta \in \mathbb{T} \setminus E \), and set \( \delta = \dist(\zeta, E) \), \( \delta' = \dist(\zeta', E) \). By Lemmas~\ref{lemma1} and Lemma~\ref{lemma2}, it suffices to prove that
\[
I = \int_\mathbb{T} \delta ^2\, \omega'(\delta)^2 \left( \int_{\zeta' \in \Gamma(\zeta, E)} \frac{\log^2(\delta / \delta')}{|\zeta - \zeta'|^2} \, |\dd\zeta'| \right) \dd\mu(\zeta)
\lesssim \int_\mathbb{T} \delta \, \omega'(\delta)^2 \, \dd\mu(\zeta).
\]
We decompose $
\Gamma(\zeta, E) = \Gamma_1(\zeta, E) \cup \Gamma_2(\zeta, E)$, 
where
\[
\Gamma_1(\zeta, E) := \left\{ \zeta' \in \mathbb{T} \setminus I(\zeta, E) : \delta' < \frac{\delta}{4} \right\}
\quad \text{and} \quad
\Gamma_2(\zeta, E) := \Gamma(\zeta, E) \setminus \Gamma_1(\zeta, E).
\]
Correspondingly, we write \( I = I_1 + I_2 \), where
\begin{eqnarray*}
I_j&=& \int_\mathbb{T} \delta ^2\, \omega'(\delta)^2 \left( \int_{\zeta' \in \Gamma _j(\zeta, E)} \frac{\log^2(\delta / \delta')}{|\zeta - \zeta'|^2} \, |\dd\zeta'| \right) \dd\mu(\zeta)\\
&=&  \int_\mathbb{T} \delta ^{2+\alpha}\, \omega'(\delta)^2  \int_{\zeta' \in \Gamma _j(\zeta, E)} \frac{\log^2(\delta / \delta')}{|\zeta - \zeta'|^2} \, |\dd\zeta'| |\dd \zeta|,, \quad j=1,2.
\end{eqnarray*}
We have
\begin{align*}
I_1 &\lesssim \int_{\mathbb{T}} \int_{\Gamma_1(\zeta, E)} \delta'^{2(1 - \tfrac{1}{\gamma})} \omega'(\delta')^2 \, \delta^{\frac{2}{\gamma} + \alpha} \frac{\log^2(\delta / \delta')}{|\zeta - \zeta'|^2} \, |\dd\zeta'| \, |\dd\zeta| \\
&\lesssim \int_{\mathbb{T}} \int_{\Gamma_1(\zeta, E)} \delta'^{2(1 - \tfrac{1}{\gamma})} \omega'(\delta')^2 (\delta - \delta')^{\frac{2}{\gamma} + \alpha} \frac{\log^2\left((\delta - \delta') / \delta'\right)}{|\zeta - \zeta'|^2} \, |\dd\zeta'| \, |\dd\zeta| \\
&\lesssim \int_{\mathbb{T}} \delta'^{2(1 - \tfrac{1}{\gamma})} \omega'(\delta')^2 \int_{|\zeta - \zeta'| > \delta'} \frac{\log^2(|\zeta - \zeta'| / \delta')}{|\zeta - \zeta'|^{2 - \tfrac{2}{\gamma} - \alpha}} \, |\dd\zeta| \, |\dd\zeta'| \\
&\lesssim \int_{\mathbb{T}} \delta'^{2(1 - \tfrac{1}{\gamma})} \omega'(\delta')^2 \int_{t \geq \delta'} \frac{\log^2(t / \delta')}{t^{2 - \tfrac{2}{\gamma} - \alpha}} \, dt \, |\dd\zeta'| \\
&\asymp \int_{\mathbb{T}} \delta'^{1+\alpha} \omega'(\delta')^2 \, |\dd\zeta'|.
\end{align*}
For \( \zeta' \in \Gamma_2(\zeta, E) \), we have $\delta/4\leq \delta '\leq \delta$ and $|\zeta - \zeta'| \geq \delta$. Then

\begin{align*}
I_2 &\lesssim \int_{\mathbb{T}} \delta^{2 + \alpha} \omega'(\delta)^2 \int_{|\zeta - \zeta'| \geq \delta} \frac{|\dd\zeta'|}{|\zeta - \zeta'|^2} \\
&\lesssim \int_{\mathbb{T}} \delta^{1 + \alpha} \omega'(\delta)^2 \, |\dd\zeta|.
\end{align*}
This completes the proof.
\end{proof}


\section[Multipliers]{\bf{Multipliers in the Classical Dirichlet space}}

\subsection{Multipliers from the Local Dirichlet Integral}

We start with the following lemma, relating Carleson measures on the circle to multipliers.
\begin{lem}\label{Dirichlet-local}
Let $f \in \mathcal{D} \cap H^\infty$. The following assertions are equivalent :
\begin{enumerate}
\item  $f$ is a multiplier for $\mathcal{D}$, 
\item The measure $|f'(z)|^2 dA(z)$  is a Carleson measure for $\mathcal{D}$,
\item The measure $\mathcal{D}_\zeta(f) |\dd\zeta|$ is a Carleson measure for $\mathcal{D}$.
\end{enumerate}
\end{lem}

\begin{proof} $(1) \iff (2)$ is trivial and well known. Let us prove the equivalence between $(1)$ and $(3)$. By Douglas's formula, we have
\begin{align}\label{eqA}
\mathcal{D}(fg) &= \frac{1}{4\pi^2} \int_\mathbb{T} \int_\mathbb{T} \frac{|f(\zeta) g(\zeta) - f(\zeta') g(\zeta')|^2}{|\zeta - \zeta'|^2} |\dd\zeta| |\dd\zeta'| \nonumber \\
&= \frac{1}{4\pi^2} \int_\mathbb{T} \int_\mathbb{T} \frac{|g(\zeta)(f(\zeta) - f(\zeta')) + f(\zeta')(g(\zeta) - g(\zeta'))|^2}{|\zeta - \zeta'|^2} |\dd\zeta| |\dd\zeta'| \nonumber.
\end{align}
Since
\[
\int_\mathbb{T} \int_\mathbb{T} |f(\zeta)|^2 \frac{|g(\zeta) - g(\zeta')|^2}{|\zeta - \zeta'|^2} |\dd\zeta| |\dd\zeta'| \lesssim \|f\|_{\infty} \mathcal{D}(g),
\]
it follows that $\mathcal{D}(fg)$ is finite if and only if
\[
\int_\mathbb{T} \int_\mathbb{T} |g(\zeta)|^2 \frac{|f(\zeta) - f(\zeta')|^2}{|\zeta - \zeta'|^2} |\dd\zeta| |\dd\zeta'| < +\infty.
\]
Thus, $f$ is a multiplier if and only if $\mathcal{D}_\zeta(f) |\dd\zeta|$ is a Carleson measure for $\mathcal{D}$.
\end{proof}

By Theorem  \ref{estimationlocal} and Lemma \ref{Dirichlet-local},  we have the following lemma
 \begin{lem}\label{CarlesonMeasure1}  Let $0<\alpha <1/2$ and let $E\subset \T$, then
 \begin{enumerate}
 \item If $E\in \mathcal{L}_2$, then $f_{\omega,E}\in \mathcal{M}(\cD)$ if and only if 
 $  d\mu_{\omega,E}(\zeta):=\dist(\zeta,E)\omega'(\dist(\zeta,E))^2 |\dd \zeta| $  is a Carleson measure of $\cD$.
\item If $E\in \mathcal{L}_1$, then $f_{\alpha,E}\in \mathcal{M}(\cD)$  if and only if 
$ d\mu_{\alpha,E}(\zeta):=\dist(\zeta,E)^{2\alpha-1} |\dd \zeta| $ is a Carleson measure of  $\cD.$
\end{enumerate}
\end{lem}



Given an arc $I\subset \mathbb{T}$, let us write $|I|$ for its arc-length, and $S(I)$ for the associated Carleson box, defined by
\[S(I) := \left\lbrace re^{i\theta} : e^{i\theta} \in I, 1-\vert I\vert < r < 1 \right\rbrace.\]
Recall that the reproducing kernel of the
Dirichlet space $\cD$ is given by
$$k_w(z) = \frac{1}{\bar wz}
\log \frac{1}{1-  \bar wz}, \qquad z,w\in \D.$$

we will use the following  characterization of Carleson measure given by Arcozzi, Rochberg and
Sawyer \cite{ARS1} and

\begin{thm}\label{ARS}\cite{ARS1}
 Let $\mu$  be a finite Borel measure on $\overline{\D}$. Then $\mu$ is a Carleson measure for
$\cD$ if  and only if there exists a positif  constant $C$ such that
$$\int_{\overline{S(I)}} \int_{\overline{S(I)}} \re \;k_w(z)\dd\mu(w)\dd\mu(z) \leq C\mu( \overline{S(I)}).$$
\end{thm}

\subsubsection*{Proof of Theorems \ref{CarlesonMeasure3}} and   \ref{CarlesonMeasure2}}. Here we consider the general setting and  the special case $\omega(t) = t^\alpha$.  By Lemma \ref{CarlesonMeasure1}  it suffices to prove that the measure 
$$d\mu_{\omega,E}(\zeta):=\dist(\zeta,E) \omega'(\dist(\zeta,E))^2 |\dd \zeta|$$  
  is a Carleson measure for $\cD$ .  By Theorem \ref{ARS}, the measure $\dd\mu_{\omega,E}(\zeta)$  is a Carleson measure for $\cD$ if and only if, for every closed arc $I$, we have 
\begin{equation}\label{equationcar}
 \int_{I} \int_{I} \log \frac{1}{|\zeta-\zeta'|}\dd\mu_{\omega,E}(\zeta) \dd\mu_{\omega,E}(\xi) \lesssim\mu_{\omega,E}(I).
 \end{equation}
then \eqref{equationcar} is satisfied, and   $\dd\mu_{\omega,E}$   is a Carleson measure, therefore $f_{\omega,E}\in \mathcal{M}(\cD)$.\\

Conversely, if the function  $f_{\omega,E}$ is a multiplier for $\cD$, then the measure $\dd\mu_{\omega,E}(\zeta) :=\dist(\zeta,E)\omega'(\dist(\zeta,E))^2 |\dd \zeta|$   is a Carleson measure, thus \eqref{equationcar} is satisfaite 
which completes the proof of Theorem  \ref{CarlesonMeasure2}.\\

To obtain a simpler, more practical condition, we recall 
a one-box characterization of Carleson measures for $\cD$, analogous to the classical characterization for Hardy spaces.

\begin{thm}\label{Onebox}\cite{EKR}
Let $\mu$ be a finite positive Borel measure on $\overline{\mathbb{D}}$ satisfying 
$$\mu(\overline{S(I)})=O(\phi(\vert I\vert)), \qquad (| I|\to 0),$$
where $\phi:\left]0,2\pi\right]\rightarrow \left]0,\infty\right]$ is an increasing function such that
$$\int_0^{2\pi}\dfrac{\phi(x)}{x}dx<\infty.$$
Then $\mu$ is a Carleson measure for $\mathcal{D}$.
\end{thm}

Let 
\[
    N_E(t) := 2 \sum_{j} 1_{\{|I_j| > 2t\}}
\] 
where $(I_j)$ denotes the family of connected components of $\mathbb{T} \setminus E$.   Note that for any measurable function   
$\Omega : (0, \pi] \to \mathbb{R}^+$, we have  
\[
    \int_{\mathbb{T}} \Omega(\operatorname{dist}(\zeta, E)) |d\zeta|  
    = \int_0^{\pi} \Omega(t) N_E(t) dt.
\]
\noindent We have $tN_E(t)\leq |E_t|$, where $E_t := \left\{\zeta \in \mathbb{T} : \dist(\zeta, E) \leq t \right\}$. Let $c(E)$  denote the logarithmic capacity of $E$, Recall that if  $E$ is  the Cantor set associated with the sequence $(\zeta_n)_{n \geqslant 0}$, then 
$$c(E)=0\iff \int_0^1 \frac{ds}{|E_s|}=+\infty\iff \sum_{n\geq 0} 2^{-n}\log(1/\pi (\zeta_1\ldots\zeta_n))=+\infty$$

 We have 
the following result
\begin{cor}\label{corcantor}
Consider a sequence \( (\zeta_n) \) with values in \( (0,1/2) \) such that  \( \limsup_n \zeta_n < 1/2 \),  and let $E = E_{(\zeta_n)}$ with positive   logarithmic capacity
then the following assertions are equivalent : 
\begin{enumerate}
\item $f_{\omega, E} \in \mathcal{D}$,
\item $f_{\omega, E} \in \mathcal{M}(\mathcal{D})$,
\item $\int_\T \dd(\zeta,E)\omega'(\dd(\zeta,E))^2 |\dd\zeta|<\infty$.
\end{enumerate}
\end{cor}

\begin{proof} Let $f_{\omega, E} \in \mathcal{D}$. Since $E \in \mathcal{K} \subset \mathcal{L}_2$, by Lemma \ref{CarlesonMeasure1}, it suffices to show that the measure \( d\mu_{\omega,E}(\zeta) := \dd(\zeta,E) \omega'(\dd(\zeta,E))^2 \) is a Carleson measure for \( \mathcal{D} \).\\

We have
\begin{equation}\label{eq1}
\mu(\overline{S(I)}) \asymp \int_I \omega'(\dd(\zeta,E))^2 \dd(\zeta,E) |\dd\zeta| = \int_0^{|I|} t \omega'(t)^2 N_{E \cap I}(t) \dd t.
\end{equation}
Note that 
$$
 \displaystyle N_{E \cap I}(t) \lesssim \frac{N_E(t)}{N_E(|I|)}.
$$
 If we set $$ \phi(s):= \int_0^{s} t \omega'(t)^2  \frac{N_E(t)}{N_E(s)} \dd t,$$
 then, the inequality (\ref{eq1}) becomes 
 $$\mu(\overline{S(I)})= O( \phi(|I|)).$$
On the other hand, using the fact that $c(E) >0$ that is $\int_0^1 \frac{ds}{|E_s|}<\infty$, we have\\
\[
\begin{array}{lll}
\displaystyle \int_0^{1} \frac{\phi(x)}{x} \dd x & = & \displaystyle \int_0^1 t \omega'(t)^2 N_E(t) \left( \int_{s=t}^1 \frac{\dd s}{s N_E(s)} \right) \dd t\\

& \asymp & \displaystyle  \int_0^1 t \omega'(t)^2 N_E(t) \dd t.\\

& \lesssim & \displaystyle  \int_0^1 \omega'(t)^2 |E_t| \dd t.\\

& \asymp & \displaystyle  \int_{\mathbb{T}} \omega'(\dd(\zeta,E))^2 \dd(\zeta,E) |\dd\zeta| < +\infty.
\end{array}
\]
The proof is completed by applying Theorem \ref{Onebox}.
\end{proof}

\subsection{Remarks and Examples}

{\bf 1.} There exists $E$ such that $f_{\alpha,E}\in \mathcal{D}\setminus \cM(D)$ for $\alpha<1/2$.
Before giving such an example, let us recall that a necessary condition for a measure $\mu$ to be a Carleson measure for the Dirichlet space $\cD$  is that
\begin{equation}\label{CN}
\mu(\overline{S(I)}) = O(1/\log(1/|I|)),  \quad (|I|\to 0)
\end{equation}

 Let us now $\theta_n=\sum_{k\geq n} t_k$ where $t_n^{2\alpha}=1/(n(\log n)^\beta)$,  where  $\beta>1$  will be chosen later. Set  
$$E=\big\{e^{i\theta_n}, n\geq 2\big\}\cup \{1\}. $$ 
We  have 
$$ \int_\mathbb{T} \dist(\zeta,E)^{2\alpha-1}{|\dd\zeta|} \asymp \sum_{n\geq 1} (\theta_n-\theta_{n+1})^{2\alpha} = \sum_{n\geq 2} \frac{1}{n(\log n)^\beta}<\infty.$$
Then $f_{\alpha,E}\in \cD$.

Now we proof that  $f_{\alpha,E}\notin\cM(D)$ for $\alpha<1/2$. Since \eqref{remark}
 $ \dist(\zeta,E)^{2\alpha-1}\lesssim \cD_\zeta(f_{\alpha,E})$. Therefore,  if $ \cD_\zeta(f_{\alpha,E})|\dd \zeta|$ is  a Carleson measure, then for all arcs $I\subset \TT$, we have 
\begin{equation}
\int_I \dist(\zeta,E)^{2\alpha-1} |\dd\zeta|  \leq \frac{C}{\log (1/|I|)}.
\end{equation}
Take $I=(1, e^{i\theta_N})$, 
$$|I|=\theta_N=\sum_{k\geq N} \frac{1}{n^{1/2\alpha}(\log en)^{\beta/{2\alpha}}}\asymp \frac{1}{n^{({1}/{2\alpha})-1}(\log en)^{\beta/{2\alpha}}}$$
and $\log (1/|I|)\asymp \log n$.  Since 
$$\int_I \dist(\zeta,E)^{2\alpha-1} |\dd\zeta|  \asymp \sum_{n\geq N} (\theta_n-\theta_{n+1})^{2\alpha} = \sum_{n\geq N}\frac{1}{n(\log n)^\beta}\asymp  \frac{1}{(\log n)^{\beta-1}}.$$
The necessary condition \eqref{CN} is satisfied when $\beta \in (1,2)$. Hence, if  $\beta>2$,  \eqref{CN} is not satisfied  $ \cD_\zeta(f_{\alpha,E})|\dd \zeta|$ is not a  a Carleson measure for $\cD$.  By Lemma \ref{CarlesonMeasure1},  $f_{\alpha,E}\notin\cM(D)$, which completes the proof.

3. Let us note that if we set
$$\beta_E=\sup\big\{\beta>0 \text { : }   \dist(\cdot ,E)^{-\beta}\in L^1(\T)\big\},$$ then by  Corollary \ref{corcantor} we get
\begin{itemize} 
\item if $\alpha>(1-\beta_E)/2$, then $f_{\alpha,E}\in \mathcal{M}(\cD)$, 
\item  if $\alpha<(1-\beta_E)/2$, then $f_{\alpha,E}\notin \cD$, 
\end{itemize}
The case $\alpha=(1-\beta_E)/2$ depends on $E$.
Indeed, if we specify the set $E$, we obtain a better result. For $\zeta \in (0, 1/2)$, let $E_\zeta$ be the perfect symmetric set of constant ratio $\zeta$.
Recall that $\text{\rm dim}(E_\zeta)$, the Hausdorff dimension of $E_\zeta$, is equal to $({\log 2})/{\log(1/\zeta)}$.
 We have
$$\int_\TT  \dist(\xi,E)^{-\beta} |\dd \xi|\asymp \sum_n2^n\int_0^{(1-\zeta)\zeta^n}t^{-\beta}dt\asymp(2\zeta^{1-\beta})^n$$
Hence 
$$\dist(\cdot,E)^{-\beta} \in L^1(\T) \iff \beta<1-\frac{\log 2}{\log(1/\zeta)}=1-\dim(E_\zeta)=\beta_E.$$
We get 
the following result
{ \begin{cor}  
Let $\xi \in (0, 1/2)$ and let $E_\xi$  be the perfect symmetric set of constant ratio $\xi$. The following assertions are equivalent:
\begin{enumerate}
\item $f_{\omega, E} \in \mathcal{M}(\mathcal{D})$,
\item $f_{\omega, E} \in \mathcal{D}$,
\item $
\displaystyle \int _0^1\omega '(t)^2t^{1-\dim E_\xi}dt<\infty.
$
\end{enumerate}
\end{cor}
}

We conclude this section with the following proposition
\begin{prop}Let  $E \subset \mathbb{T}$ be a closed subset such that $\dist(\cdot,E)^{-\sigma}\in L^1(\T)$, where $\sigma\in (0,1)$.
If the measure 
$$d\widetilde{\mu}_{\omega,E}:=\dist(\zeta,E)\omega '(\dist (\zeta, E))^2\log^2\big(1/\dist(\zeta,E)\big)|d\zeta|$$

 is a Carleson measure for $\mathcal{D}$, then the outer function $f_{\omega,E}$ is a multiplier for $\mathcal{D}$.
\end{prop}

\section{\bf{Approximation in Dirichlet spaces}}

Let $\mu$ a positive Borel
measure on $\T$.  Recall that the polynomials are dense in \( \cD(\mu) \).  A function \( f \in \cD(\mu) \) is called \textit{cyclic} in \( \cD(\mu) \) if 
\[
[f]_{\cD(\mu)} = \cD(\mu),
\]
where 
\[
[f]_{\cD(\mu)} := \overline{\{ p f \ : \ p \text{ is a polynomial} \}}.
\] 
In order to introduce the capacity associated with  \( \cD(\mu) \), we first recall the definition of the harmonic space associated with $\cD(\mu)$. Namely, let \( \mathcal{D}^h(\mu) \) be the set of functions \( f \in L^2(\mathbb{T}) \) such that   \(  \mathcal{D}_\mu(f)<\infty \), equipped with the norm
\[
\|f\|^2_{\mu} := |f(0)|^2 + \mathcal{D}_{\mu}(f).
\]

We recall  the definition of capacity in the sense of Beurling–Deny: let \( U \) be an open subset of $\T$, the \( c_\mu \)-capacity of \( U \) is defined by
\[
c_\mu(U) := \inf \left\{ \|f\|^2_\mu : u \in \mathcal{D}^h(\mu), f \geq 0 \text{ and } f \geq 1 \text{ a.e. on } U \right\}. 
\]
As usual, we define the \( c_\mu \)-capacity of any subset \( F \subset \mathbb{T} \) by
\[
c_\mu(F) = \inf \left\{ c_\mu(U) : U \text{ open, } F \subset U \right\}.
\]
\subsection{Polar sets} 
In this section, we establish a sufficient condition ensuring that a set has zero capacity.  

The following change-of-variables formula will be used several times in the sequel, we suppose that $\mu(E)=0$.
\[
\int_{\mathbb{T}} \Omega(\dist(\zeta,E))\, d\mu(\zeta)
   = \int_0^{2\pi} \Omega(t)\, d\mu(E_t).
\]

\begin{itemize}
\item
If \(\Omega\) is differentiable, then
\begin{eqnarray*}
\int_0^{2\pi} \Omega(t)\, d\mu(E_t)
  & =& -\int_0^{2\pi} \Big( \int_{t}^{2\pi} \Omega'(s)\, ds \Big)\, d\mu(E_{t})
     +\Omega(2\pi)\,\mu(E_{2\pi})\\
  & = &-\int_0^{2\pi} \Omega'(s)\, \mu(E_s)\, ds + \Omega(2\pi)\,\mu(\T).
\end{eqnarray*}
Hence
\begin{equation}\label{changement1}
\int_{\mathbb{T}} \Omega(\dist(\zeta,E))\, d\mu(\zeta)
   =-\int_0^{2\pi} \Omega'(s)\, \mu(E_s)\, ds + \Omega(2\pi)\,\mu(\T).
\end{equation}

\item
If \(\Omega(2t) \asymp \Omega(t)\), then
\begin{eqnarray*} \int_0^{2\pi} \Omega(t)\, d\mu(E_t)&\asymp&\sum_n \Omega(2^{-n})(\mu(E_{2^{-n}})-\mu(E_{2^{-n-1}}))\\
&\lesssim &\int_0^{2\pi} \Omega(t)\, \frac{\mu(E_t)}{t}\, dt +\Omega(2\pi)\,\mu(\T).
\end{eqnarray*}
Hence 
\begin{equation}\label{changement2}
\int_{\mathbb{T}} \Omega(\dist(\zeta,E))\, d\mu(\zeta)
   \lesssim \int_0^{2\pi} \Omega(t)\, \frac{\mu(E_t)}{t}\, dt + \Omega(2\pi)\,\mu(\T).
\end{equation}
\end{itemize}

We need the following lemma, which provides an estimate for $\cD_\mu(f_{\omega,E})$ under the assumption that $\omega$ is non-increasing.
\begin{lem}\label{lemma11}
Let \(E \subset \mathbb{T}\) be an \(\mathcal{L}_2\)-set. Let \(\omega\) be a non-increasing function such that \( \omega(t^\gamma)\) is convexe for some \(\gamma > 2\), and such that \(t|\omega'(t)|\omega(t)\) is increasing. If
\begin{equation}\label{converge1}
\int_0^1 \omega'(t)^2\, \mu(E_t)\, dt < \infty,
\end{equation}
then \(f_{\omega,E} \in \mathcal{D}(\mu)\).
\end{lem}
\begin{proof}
The proof closely follows that of Lemma \ref{lemma1} and Lemma \ref{estimationgamma}, with only minor adjustments to take into account the assumptions on $\omega$. 
Set \(\delta=d(\zeta,E)\), we obtain 
$$\mathcal{D}_{\zeta}(f_{\omega,E})\lesssim \delta \omega'(\delta)^2$$
Now let $\Omega(t)=t\omega'(t)^2$, we have $\Omega(2t)\asymp \Omega(t)$ and by \eqref{changement2},  we get 
$$\int_\TT \mathcal{D}_{\zeta}(f_{\omega,E})\, d\mu(\zeta) \lesssim  \int_\TT \Omega(\dist(\zeta,E))\, d\mu(\zeta) \lesssim \int_\TT \Omega(t)\, \frac{\mu(E_t)}{t} dt $$
and this completes the proof.
\end{proof}

The following result follows immediately from the weak inequality on capacities.
\begin{prop}\label{polair}
Let $F$ be a closed subset of  $\mathbb{T}$. Then $F$ is $c_\mu$-polar that is,
$c_\mu(F) = 0$, if and only if there exists an outer function  $f \in \mathcal{D}(\mu)$
 such that $ f^*(\zeta) = \infty $  for all $ \zeta \in F$.
\end{prop}

\begin{thm}\label{polair2}
Let \(E \subset \mathbb{T}\) be an \(\mathcal{L}_2\)-set. Assume moreover that the measure \(\mu\) satisfies
\[
\mu(E_t) = O\bigl(h(t)\bigr)
\quad\text{as }t \to 0^+,
\]
where \(h\colon(0,1)\to\R^+\), ${h(t)}/{t}\text{ is decreasing}$, $h(t)/t^\beta$ is increasing for $\beta\in (0,1)$. If
\begin{equation}\label{diverge}
\int_{0}^{1} \frac{dt}{h(t)} = +\infty,
\end{equation}
then $c_\mu(E)=0$
\end{thm}
 \begin{proof}Let $f_{\omega,E}$ be the outer function, such that 
$$\omega(t)=\log   \int_{t}^{\pi} \dfrac{\mathrm{d}s}{ h(s)}\quad \text{a.e. on } \mathbb{T}$$
The function $\omega$ is non-increasing,   the assumptions on $h$ imply that $\omega(t^\gamma)$ is concave function for some $\gamma>2$ and $t|\omega'(t)|\omega(t)$ is increasing.  
Since $h$ satifies \eqref{diverge}, then 
$$\int_0^1\frac{dt}{h(t)\Big(\int_t^1 ds/h(s)\Big)^2} <\infty.$$
Hence \eqref{converge1} is satifies. Therefore 
$f_{\omega,E}\in \cD(\mu)$ and  $f_{\omega,E}(\zeta)=\infty$ on  $E$, then by Proposition \ref{polair} $c_\mu(E)=0$.
\end{proof}

\subsection*{Remarks} 1. Since the shift operator is a $2$-isometrie in $\cD(\mu)$,  then  $\cD(\mu)\subset L^2(\mu)$. Hence  if  $\mu(E)>0$ then $c_\mu(E)>0$.\\

2. Roughly speaking, the Theorem \ref{polair2} says that if 
$$\int_0^{2\pi} \frac{dt}{\mu(E_t)}=+\infty$$
then $c_\mu(E)=0$. When $\mu = m$ is the Lebesgue measure, the capacity  $c_\mu$ is comparable to the logarithmic capacity. The condition
$$\int_0^t \frac{dt}{|E_t|} = +\infty,$$
implies $c(E) = 0$. This result is due to Carleson \cite[Theorem 2, p.~30]{C}, see also \cite[Theorm 5.4]{EEK}  on polar sets. \\

3.  Clearly, if f is cyclic in $\mathcal{D}(\mu),$ then f is outer function and its zero set  $\mathcal{Z}_{\mathbb{T}}(f)\subset \T$, is polar, i.e.
$c_\mu\big(\mathcal{Z}_{\mathbb{T}}(f)\big) = 0,$ where
\[
\mathcal{Z}_\T(f) = \left\{ \zeta \in \T : \inf_{r \to 1} |f(r\zeta)| = 0 \right\},
\]
Theorem \ref{thm:4.3} below on cyclicity allows us to recover Theorem \ref{polair2}.

\subsection{Cyclicity in $\cD(\mu)$}
 \begin{thm}
  \label{thm:4.3}
    Let $f \in \cD(\mu) \cap A(D)$ be an outer function and $E = \{\zeta \in \mathbb{T} \mid f(\zeta) = 0\}$.
    If there exists a function $g \in \cD(\mu)$ such that 
    \[
        |g(z)| \lesssim \dist(z,E)^4, \quad z \in \DD,
    \]
    then $g \in [f]_{\cD(\mu)}$.
\end{thm} 
\begin{proof} See \cite {EKR,EKR2,EEKAH}
\end{proof}
We require the following lemma, which provides an estimate for \( \cD_\mu(f_{\omega,E}) \) under weaker regularity assumptions, in particular, it does not require the monotonicity of \( t\omega'(t)/\omega(t) \), provided that $E$ is a $\mathcal{K}$-set. 
\begin{lem}\label{lemma11}
Let \( E \subset \mathbb{T} \) be a \(\mathcal{K}\)-set.  Let \(\omega\) be  an increasing function such that \(\omega(t^\gamma)\) is concave for some \(\gamma > 2\) and   $\omega(t)/t^\varepsilon$ is non-increasing for some $\varepsilon \in (0,\beta_E)$. Let $\mu$ a positive Lebesgue measure such that  \(\mu(E_t) = O(h(t))\) as \(t \to 0\), where \(h\) is increasing and \(h(t)/t\) is decreasing. Then 
\begin{equation}\label{converge1} 
\cD_\mu(f_{\omega,E})\lesssim  \|\omega'h\|_\infty \|\omega\|_\infty.
\end{equation}
where the implied constant depends only  $\gamma$,  $\varepsilon$ and $\beta_E$.
\end{lem}
\begin{proof}
The proof closely follows that of Lemma~\ref{lemma1} and Lemma \ref{estimationgamma}, with only minor modifications. We therefore highlight only the steps where these changes occur. 
Set \(\delta=d(\zeta,E)\) and \(\delta'=d(\zeta',E)\).  For the integrals $I_i$ ($i=1,2,3$) in Lemma \ref{lemma1}, only $I_1$ involves the condition that $t w'(t)/w(t)$ is increasing. Therefore, the only change concerns $I_1$.  
Let 
$$\Omega(u)=\frac{1}{u^2} \int_{t=0}^u t \frac{\omega'(t)}{\omega(t)} \int_{s=t}^\delta \omega'(s)\omega(s)\, ds\, dt=\frac{1}{u^2} \int_{s=0}^u\omega'(s)\omega(s)\int_{t=0}^s t \frac{\omega'(t)}{\omega(t)}\, ds\, dt$$
Hence 
$-\Omega'(u)\leq \Omega(u)/u$.
From \eqref{eqI1}  and by \eqref{changement2}, we have
$$\int_\TT \mathcal{D}_{\zeta,I_1}(f_{\omega,E})\,d\mu(\zeta)
\asymp \int_\TT \Omega(\dist(\zeta,E))\dd\mu(\zeta)= -\int_0^1 \Omega'(t){\mu(E_t)}\dd t +\Omega(2\pi)\mu(\T).$$
Hence 
\begin{align*}
\int_\TT \mathcal{D}_{\zeta,I_1}(f_{\omega,E})\,d\mu(\zeta)
&\leq \int_0^1 \frac{1}{u^2} \int_0^u t \frac{\omega'(t)}{\omega(t)} \int_t^u \omega'(s)\omega(s)\, ds\, dt \, \frac{h(u)}{u} du \\
&= \int_0^1 t \frac{\omega'(t)}{\omega(t)} \int_t^1 s^{1-1/\gamma} \omega'(s)\omega(s) \frac{ds}{s^{1-1/\gamma}} \int_s^1 \frac{h(u)}{u^3} du \\
&\leq \int_0^1 t^{2-1/\gamma} \frac{\omega'(t)^2}{\omega(t)} \int_t^1 \omega(s)  h(s)\, \frac{ds}{s^{3-1/\gamma}} \\
&\leq\int_0^1 \omega'(t)^2 h(t)\, dt.
\end{align*}

Now, concerning $\Gamma$,  let $I_n =\{\zeta'\in \T \text{ : }   |\zeta-\zeta'|< 2^{n} \delta\}$.  Let $\beta =(\varepsilon+ \beta_E)/2$, since $E$ is a  \(\mathcal{K}\)-set, by (5) Lemma \ref{kset},  $| \{ \zeta' \in I_n \text{ : } \dist(\zeta', E)<t \}| \leq |I_n|(t/|I_n|)^{\beta}$. We get 
\begin{align*}
\int_{|\zeta-\zeta'|>\delta, \, \delta'<t} \frac{|\mathrm{d}\zeta'|}{|\zeta-\zeta'|^2} 
&= \sum_{n\ge 1} \int_{I_{n+1}\setminus I_n, \; \delta'<t} \frac{|\mathrm{d}\zeta'|}{|\zeta-\zeta'|^2} \\
&\le \sum_{n\ge 1} \frac{1}{(2^n \delta)^2} \int_{ I_n, \, \delta'<t} |\mathrm{d}\zeta'| \\
& \le \sum_{n\ge 1} \frac{1}{(2^n \delta)^2} | \{ \zeta' \in I_n \text{ : } \dist(\zeta', E)<t \} | \\
&\lesssim \sum_{n\ge 1} \frac{1}{(2^n \delta)^2} 2^{n+1} \delta \left( \frac{t}{2^{n+1}\delta} \right)^{\beta}\asymp 
 \frac{t^{\beta}}{\delta^{1+\beta}}.
\end{align*}
Hence, by \eqref{eqgamma}, we obtain
\begin{align*}
 \mathcal{D}_{\zeta,\Gamma(\zeta,E)}(f_{\omega,E})\,
&= 4  \int_{|\zeta-\zeta'|>\delta, \, \delta'<\delta} \left( \int_{\delta'}^\delta \frac{\omega'(t)}{\omega(t)} \int_t^\delta \omega'(s)\omega(s)\, ds\, dt \right) \frac{|\mathrm{d}\zeta'|}{|\zeta-\zeta'|^2}  \\
&\asymp  \int_0^\delta \frac{\omega'(t)}{\omega(t)} \int_t^\delta \omega'(s)\omega(s)\, ds \int_{|\zeta-\zeta'|>\delta, \, \delta'<t} \frac{|\mathrm{d}\zeta'|}{|\zeta-\zeta'|^2}\notag \\
&\le \frac{1}{\delta^{1+\beta}}  \int_0^\delta t^{\beta_E} \frac{\omega'(t)}{\omega(t)} \int_t^\delta \omega'(s)\omega(s)\, ds dt = \widetilde{\Omega}(\delta)
\end{align*}
Note that ${\widetilde{\Omega}}^{'}(u)\leq \widetilde{\Omega}(u)/u$, then  by \eqref{changement2} we get 
\begin{align*} 
\int_\TT  \mathcal{D}_{\zeta,\Gamma(\zeta,E)}(f_{\omega,E})d\mu(\zeta)&\lesssim \int_{u=0}^1 \int_{t=0}^u t^{\beta} \frac{\omega'(t)}{\omega(t)} \int_{s=t}^\delta \omega'(s)\omega(s)\, \frac{h(u)}{u^{2+\beta}} ds dt du \notag \\
&\lesssim \int_{t=0}^1t^{\beta}\frac{\omega'(t)}{\omega(t)}  \int_{s=t}^1 \omega'(s)\omega(s)\, \frac{h(s)}{s^{1+\beta}} \, ds dt \notag \\
&\lesssim \|\omega'h\|_\infty  \int_{t=0}^1t^{\beta}\frac{\omega'(t)}{\omega(t)}  \int_{s=t}^1
 \frac{\omega(s)}{s^\varepsilon}\frac{ds}{s^{1+\beta-\varepsilon}}  dt \notag \\
& \asymp  \|\omega'h\|_\infty \|\omega\|_\infty.
\end{align*}
The proof is now complete.
\end{proof}

\subsection*{Proof of Theorem \ref{thm:4.3}} 
Write $\sigma=1-1/\gamma_0$, let $\gamma\in (0,\gamma_0)$, $\alpha\in (0,1/\gamma_0)$,  $\sigma\in (0,1/2)$ and $\eta \in (0,1)$. Consider the outer function \( f_{\omega_\eta,E} :[0,\pi]\to. \RR_+\) where
 \( \omega_\eta \)  is  given by
\[
\omega_\eta(t) =
\begin{cases} 
\displaystyle c_\eta t^\alpha & \text{if } 0 \leq t \leq \eta, \\
&\\
\displaystyle  \left(\int_{t}^{2\pi} \dfrac{\mathrm{d}s}{ h(s)}\right)^{-\sigma}& \text{if } \eta < t \leq \pi
\end{cases}
\]
here $c_\eta$  is the constant chosen to make $\omega_\eta$ continuous.  Since $\int_0^{2\pi} ds/h(s)=\infty$ and \(h(t)/t^{1-1/\gamma_0}\) is increasing, we can choose $\eta_0\in (0,\pi)$ such that for $\eta\in (0,\eta_0)$ 
\begin{enumerate}
\item $\omega_\eta$ is increasing,
\item $\omega_\eta(t)/t^\varepsilon$ is non-increasing for any $\varepsilon>\alpha$, 
 \item \(\omega_\eta (t^\gamma)\) is concave,  
 \item 
$\displaystyle \sup_{\eta <\eta_0}  \| \omega_\eta 'h \|_\infty\|\omega_\eta\|_\infty<\infty.$
\end{enumerate}
By Lemma~\eqref{lemma11}, we have \( f_{\omega_\eta,E} \in \mathcal{D}(\mu) \).
Note that \( |f_{\omega_\eta,E}^*| \to |g^*| \) a.e. on \( \mathbb{T} \) as \( \eta \to 0 \), where \(g\) is the outer function defined by
\[
g^*(\zeta)
   = \left( \int_{\dist(\zeta,E)}^{2\pi} \frac{ds}{h(s)} \right)^{-\sigma},
   \qquad \text{a.e. on } \mathbb{T}.
\]
It is clear that \( \log |g^*| \in L^1(\mathbb{T}) \).
By part (2) of Theorem~\ref{appartdmu}, we obtain \( g \in \mathcal{D}(\mu) \), and by Lemma~\ref{lemma11}, we also have \( 1/g \in \mathcal{D}(\mu) \).
Hence \( g \) is cyclic~\cite{EEL}.

By Theorem~\ref{thm:4.3}, we have \( f_{\omega_\eta,E}^{4/\alpha} \in [f]_{\mathcal{D}(\mu)} \).
Therefore, by the Richter--Sundberg Theorem \cite{RSJOT}, it follows that
\( f_{\omega_\eta,E} \in [f]_{\mathcal{D}(\mu)} \).
Consequently, \( g \in [f]_{\mathcal{D}(\mu)} \).
Since \( g \) is cyclic in \( \mathcal{D}(\mu) \), we conclude that
\[
[f]_{\mathcal{D}(\mu)} = \mathcal{D}(\mu),
\]
which completes the proof.


\subsection*{Acknowledgements}
The first named author was partially supported by the Arab Fund Foundation Fellowship Program. The Distinguished Scholar Award - File 1092. 
The second author was partially supported by ANR Project ANR-24-CE40-5470.


\bibliography{samplebib}
\end{document}